\documentclass[12pt]{article}
\usepackage[margin = 1in]{geometry}
\usepackage{amsfonts,latexsym,amsmath,amssymb,amsthm}
\usepackage[mathscr]{eucal}
\usepackage{graphicx,color}
\usepackage{comment}
\usepackage{hyperref}
\usepackage{epstopdf}
\newtheorem{theorem}{Theorem}
\newtheorem{lemma}{Lemma}
\newtheorem{proposition}{Proposition}

\newtheorem{definition}{Definition}

\newtheorem{remark}{Remark}

\newtheorem{example}{Example}
\newtheorem{hypothesis}{Hypothesis}

\newcommand*\Laplace{\mathop{}\!\mathbin\bigtriangleup}

\newcommand{\RR}{\ensuremath{\mathbb{R}}}
\def\etA{e^{tA}}

\def\sat{\texttt{sat}}

\def\satl{\mathfrak{sat}_{L^\infty(0,1)}}

\begin{document}
\title{ \LARGE \bf%
Stability analysis of dissipative systems subject to nonlinear damping via Lyapunov techniques\footnote{This research was partially supported by the iCODE Institute, research project of the IDEX Paris-Saclay, and by the Hadamard Mathematics LabEx (LMH) through the grant number ANR-11-LABX-0056-LMH in the "Programme des Investissements d'Avenir", and by the European Research Council (ERC) through an ERC-Advanced Grant for the TAMING project (grant agreement 66698).}
}

\author{Swann Marx$^{1}$, Yacine Chitour$^{2}$ and Christophe Prieur$^{3}$}

\thanks{$^{1}$Swann Marx is with LAAS-CNRS, Universit\'e de Toulouse, CNRS, 7 avenue du colonel Roche, 31400, Toulouse, France
        {\tt\small marx.swann@gmail.com}.}%
\thanks{$^{2}$Yacine Chitour is with  Laboratoire des Signaux et Syst\`emes
(L2S), CNRS - CentraleSupelec - Universit\'e Paris-Sud, 3, rue Joliot Curie, ´
91192, Gif-sur-Yvette, France,
{\tt\small yacine.chitour@lss.supelec.fr}.}%
\thanks{$^{3}$Christophe Prieur is with Univ. Grenoble Alpes, CNRS, Gipsa-lab, F-38000 Grenoble, France,
   {\tt\small christophe.prieur@gipsa-lab.fr}.}

\maketitle

\abstract{In this article, we provide a general strategy based on Lyapunov functionals to analyse global asymptotic stability of linear infinite-dimensional systems subject to nonlinear dampings under the assumption that the origin of the system is globally asymptotically stable with a \emph{linear damping}.
%, we propose to build a Lyapunov functional when modifying the damping with a nonlinearity. 
To do so, we first characterize, in terms of Lyapunov functionals, several types 
of asymptotic stability for linear infinite-dimensional systems, namely the exponential and the polynomial stability. Then, we derive a Lyapunov functional for the nonlinear system, which is the sum of a Lyapunov functional coming from the linear system and another term with compensates the nonlinearity. Our results are then applied to the linearized Korteweg-de Vries equation and the 1D wave equation. }

\section{Introduction}

This paper is concerned with the asymptotic behavior analysis of infinite-dimensional systems subject to a nonlinear damping. These systems are composed by abstract operators generating a strongly continuous semigroup of contractions and a bounded operator representing the control operator (see e.g., \cite{tucsnak2009observation} or \cite{miyadera1992nl_sg} for the introduction of linear and nonlinear operators generating semigroups, respectively). These systems might be for instance a hyperbolic PDE, or a parabolic one or even the linearized Korteweg-de Vries equations. Assuming that a linear damping renders the origin of these systems globally asymptotically stable, we propose a general strategy to analyze the asymptotic behavior of these systems when modifying the linear damping with a nonlinearity.
In contrast with the existing litterature, which uses either integral inequalities(see e.g. \cite{alabau1999stabilisation}, \cite{alabau2002indirect}, \cite{alabau2012some}, \cite{MVC}) or a frequential approach (cf. \cite{jayawardhana2008infinite}, \cite{curtain2004absolute}) or even a compactness uniqueness strategy (\cite{enrike1990exponential}, \cite{rosier2006global}, \cite{mcpa2017siam}), we propose here to design Lyapunov functionals to characterize our results, extending to the infinite-dimensional setting a strategy first deviced in \cite{liu1996finite} for finite-dimensional systems. %This fucntionnal is inspired by some results from finite-dimensional systems \cite{liu1996finite}. 

\paragraph{Lyapunov functionnals for infinite-dimensional linear systems} 
%The construction of a Lyapunov functional for finite-dimensional systems subject to a nonlinear damping relies on the linear system, i.e., when the damping is not modified by a nonlinearity. Indeed, the latter Lyapunov functional is composed by two terms: the first one is the same than in the linear damping case and the second one is added in order to compensate the nonlinearity of the system. In this paper, we propose some  
In the case where the origin of the linear system is \emph{globally exponentially stable}, there exists a direct way to construct the Lyapunov functional. It relies mainly on the result provided in \cite{DATKO1970610}. However, it is known that an equilibrium point for an infinite-dimensional system that is globally asymptotically stable is not necessarily exponentially stable. In some cases, this point is only \emph{polynomially stable}, i.e., trajectories of the system converge with a decay rate expressed as $\frac{1}{(1+t)^\gamma}$, where $\gamma$ is a positive constant. 

Most of the existing litterature analyzes this asymptotic behavior with some integral inequalities \cite{russell1975decay}, \cite{alabau2002indirect} or with a frequential approach \cite{liu2005characterization}. In constrast with these papers, we propose here to construct a Lyapunov functional in the case of polynomial stability. At the best of our knowledge, such a result is new. Note moreover that it is crucial in our approach, since this functional will be used in the case where the damping is modified with a nonlinearity.

\paragraph{Nonlinear damping for infinite-dimensional systems} There exist many works dealing with nonlinear damping for infinite-dimensional systems. Some of them tackle specific PDEs as for instance hyperbolic ones (see e.g., \cite{haraux1985stabilization}, \cite{MVC} or \cite{alabau2012some}) and others propose a general framework using abstract operators (see \cite{slemrod1989mcss}, \cite{seidman2001note}, \cite{lasiecka2002saturation} and \cite{curtain2016stabilization} for a specific case of nonlinear damping, namely the \emph{saturation}). These papers, which deal with abstract operators, usually assume that the space where the damping takes value, namely $S$, is the same as the control space, namely $U$. However, in practice, this is not the case.

In constrast with existing works for abstract control systems, we aim here at giving a general definition of nonlinear dampings when the nonlinear damping space $S$ is not necessarily equal to the control space $U$. With such a formalism, we are able to make a link between the litterature on abstract operators and the one on hyperbolic systems. At the best of our knowledge, this formalism has been introduced first in \cite{mcp2018ECC} in the case where the nonlinear damping is a saturation.  

In many works, specific PDEs subject to a nonlinear damping have been studied. In \cite{lasiecka2002saturation}, the origin of a wave equation subject to a nonlinear damping, either distributed or located at the boundary, has been proved to be globally asymptotically stable, in the case $S = U$. In \cite{prieur2016wavecone}, a similar result has been stated, but in the case where $S\neq U$. In \cite{daafouz2014nonlinear}, the global asymptotic stability of a PDE coupled to an ODE with a saturated feedback law at the boundary has been tackled. There exist also some papers dealing with local asymptotic stability (see \cite{kang2018boundary} or \cite{kang2017boundary}). Note that both situations ($S=U$ and $S\neq U$) have been tackled for the specific nonlinear partial differential equation Korteweg-de Vries equation in \cite{mcpa2017siam}, in the case where the damping is a saturation. 

\paragraph{Contribution} In this paper, we study two cases: either the origin of the infinite-dimensional system with a linear damping is globally exponentially stable or it is globally polynomially stable. In both cases, we derive a Lyapunov functional which allows us to prove and even characterize the decay rate of the trajectories. 

In the first case (i.e., the origin of the linear system is globally exponentially), we derive a strict and \emph{global} Lyapunov function if $S=U$. By global, we mean that the Lyapunov function does not depend on the initial condition, neither the decay rate. However, if $S\neq U$, we are not able to obtain such a result, but we prove that the origin of the system is semi-globally exponentially stable, meaning in particular that the decay rate of the trajectories depends on the initial condition.

In the second case (i.e., the origin of the linear system is globally polynomially stable), only in the case where $S=U$, we prove that the origin of the system is semi-globally polynomially stable. As in the exponential case, this means that the decay rate of the trajectories depends on the initial condition.

\paragraph{Outline} Section \ref{sec_lyap} provides some necessary and sufficient conditions in terms of Lyapunov functionals for infinite-dimensional systems. In particular, we provide a new Lyapunov functional in the case where the origin is globally polynomially stable. In Section \ref{sec_main}, nonlinear dampings for infinite-dimensional systems are introduced and our main results are stated. Their proofs are then given in Section \ref{sec_proof}. These results are illustrated in Section \ref{sec_example} on some examples, namely the linearized Korteweg-de Vries equation and the 1D wave equation. Section \ref{sec_conclusion} collects some concluding remarks and further research lines to be investigated. Appendix \ref{app-finite} tackles the specific case of finite-dimensional systems and provides also a decay rate characterization, that applies also for the case $S=U$ and the linear damping stabilizes exponentially the system. 

\paragraph{Acknowledgements:} The authors want to warmly thank Fatiha Alabau-Boussouira and Enrique Zuazua  for all the encouraging and interesting discussions and for having pointed out a large number of crucial references. We would like to thank also Nicolas Burq for the interest in the present work.

\section{Lyapunov criteria for linear infinite-dimensional systems}
\label{sec_lyap}

Let $H$ be a real Hilbert space equipped with the scalar product $\langle \cdot,\cdot\rangle_H$. Let $A:D(A)\subset H\rightarrow H$ be a (possibly unbounded) linear operator whose domain $D(A)$ is dense in $H$. We suppose that $A$ generates a strongly continuous semigroup of contractions denoted by $(e^{tA})_{t\geq 0}$.  We use $A^\star$ to denote the adjoint operator of $A$. 

In this section, we consider the linear system given by
\begin{equation}
\label{linear-system}
\left\{
\begin{split}
&\frac{d}{dt}z = Az,\\
&z(0)=z_0.
\end{split}
\right.
\end{equation}
Since $A$ generates a strongly continuous semigroup of contractions, there exist both strong and weak solutions to \eqref{linear-system}. Moreover, the origin of \eqref{linear-system} is Lyapunov stable\footnote{The origin of \eqref{linear-system} is said to be Lyapunov stable in $H$ if, for any positive $\delta$, there exists a positive constant $\varepsilon=\varepsilon(\delta)$ such that
$$
\Vert z_0\Vert_H\leq \varepsilon\Rightarrow \Vert \etA z_0\Vert_H\leq \delta.
$$
} in $H$. Indeed, the property of contraction satisfied by $(\etA)_{t\geq 0}$ implies that
\begin{equation}
\Vert \etA z_0\Vert_H\leq \Vert z_0\Vert_H. 
\end{equation}
The origin is attractive in $H$ if, for every $z_0\in H$, one has
\begin{equation}
\lim_{t\rightarrow +\infty}\Vert \etA z_0\Vert_H=0, 
\end{equation}
and this property is also referred as \emph{strong stability} (see e.g., \cite[Section 1.3]{alabau2012some}). 
This section aims at characterizing the decay rate of the trajectory when assuming that the origin is attractive. We first consider \emph{global exponential stability}:
\begin{definition}[Global exponential stability]
The origin of \eqref{linear-system} is said to be \emph{globally exponentially stable} if there exist two positive constants $C$ and $\alpha$ such that, for any $z_0\in H$,
\begin{equation}\label{eq:GES}
\Vert \etA z_0\Vert_H \leq C e^{-\alpha t}\Vert z_0\Vert_H,\qquad \forall t\geq 0.
\end{equation} 
\end{definition}

\begin{remark}\label{rem:DA}
If the origin is of \eqref{linear-system} is globally exponentially stable in $H$, then, provided that the initial condition $z_0$ is in $D(A)$, the origin is also globally exponentially stable in $D(A)$. Indeed, since $A$ generates a strongly continuous semigroup of contractions, then, for any initial condition $z_0\in D(A)$, $Ae^{tA}z_0\in H$, for all $t\geq 0$. Since \eqref{eq:GES} holds, this means in particular that 
\begin{equation*}
\Vert e^{tA} Az_0\Vert_H \leq C e^{-\alpha t} \Vert A z_0\Vert_{H},\quad \forall t\geq 0.
\end{equation*}
Note moreover that $e^{tA}A=Ae^{tA}$ (see e.g., \cite[Proposition 2.1.5]{tucsnak2009observation}) and $\Vert \cdot \Vert_{D(A)}:=\Vert \cdot \Vert_H + \Vert A\cdot\Vert_H$. Therefore,
\begin{equation*}
\Vert e^{tA}z_0\Vert_{D(A)} \leq C e^{-\alpha t}\Vert z_0\Vert_{D(A)},\quad \forall t\geq 0.
\end{equation*}
\end{remark}

Another characterization of attractivity is the \textit{polynomial stability}. There exists several possible definitions referring to polynomial stability in the litterature, cf. the  nice survey \cite{alabau2012some}. In any case, this is a weaker notion of attractivity than exponential stability, because the initial condition usually belongs to a more regular space defined as follows:
\begin{equation}
D(A^\theta) = \lbrace z\in H \mid A^i z\in H,\: i=1,\ldots,\theta \rbrace,
\end{equation}
where $\theta$ is a positive integer. %\in\mathbb{N}\setminus\lbrace 0\rbrace$. 
We suppose with no further mention in the remaining sections of the paper that, every time polynomial stability is at stake,  then $A^{\theta}$ is well-defined and $D(A^\theta)$ is dense in $H$ for some positive integer $\theta$. 

This space is endowed with the following norm
\begin{equation}
\Vert \cdot \Vert_{D(A^\theta)}:=\sum_{i=0}^\theta \Vert A^i \cdot \Vert_{H}.  
\end{equation}
By $A^i z$, we mean that $A$ is applied $i$ times to $z$. We define $A^0 = I_H$ so that we retrieve the classical definition of the graph norm of the operator $A$. This definition is borrowed from \cite{tucsnak2009observation}, just above Proposition 2.2.12. 

\begin{definition}[Global polynomial stability]
\label{def-pol-stab}
Given $\theta$ a positive integer, the origin of \eqref{linear-system} is said to be polynomially stable if there exist two positive constants $C$ and $\gamma:=\gamma(\theta)$ such that, for any initial condition $z_0\in D(A^\theta)$,
\begin{equation}
\label{poly1}
\Vert \etA z_0\Vert_H \leq \frac{C}{(1+t)^{\gamma}} \Vert z_0\Vert_{D(A^\theta)},\qquad \forall t\geq 0.
\end{equation}
\end{definition}
%\begin{remark}\label{rem:SG}
%As we will see later, one can weaken the definitions of exponential stability and $\theta$-Polynomial stability
%by making the constants $C$ and (or) $\alpha$ depend on the initial condition $z_0$.
%\end{remark}
In recent decades, Lyapunov functions have been instrumental to characterize stability for either finite-dimensional or infinite-dimensional systems. The main result of \cite{DATKO1970610} is stated in the following proposition. 
\begin{proposition}[Exponential stability \cite{DATKO1970610}]
\label{lyapunov90}
The origin of \eqref{linear-system} is said to be globally exponentially stable if and only if there exist a self-adjoint, positive definite and coercive operator $P\in\mathcal{L}(H)$ and a positive constant $C$ such that
\begin{equation}
\label{lyapunov-linear-equation}
\langle A z,Pz\rangle_H + \langle P z, Az\rangle_H\leq-C\Vert z\Vert^2_H,\quad \forall z\in D(A).
\end{equation}
One can choose $P$ in the latter equation in the form 
\begin{equation}
P=\int_0^\infty e^{sA^\star} e^{sA} ds+\alpha I_H,
\end{equation}
with $\alpha>0$. Note that this operator defines also a bounded operator of $D(A)$.
\end{proposition}
%The last part of the proposition is a consequence of Remark~\ref{rem:DA}. 
Note that an operator $P$ satisfying \eqref{lyapunov-linear-equation} has also been considered in the context of the asymptotic stability analysis of linear switched systems in \cite{hante2011converse}. 

We next turn to a similar characterization of polynomial stability, i.e., in terms of a Lyapunov function. 
To the best of our knowledge, polynomial stability seems to have first been considered in \cite{russell1975decay} and \cite{alabau1999stabilisation}. Later on, it has been studied with spectrum analysis in  \cite{liu2005characterization}. We propose a Lyapunov characterization of such a stability with the following proposition.

\begin{proposition}[$\theta$-global polynomial stability]
\label{lyapunov-2018}
Given $\theta$ a positive integer, the origin of \eqref{linear-system} is said to be globally polynomially stable with $\gamma>\frac{1}{2}$ if and only if there exist a self-adjoint, positive-definite and coercive operator $P_\theta : D(A^\theta)\rightarrow D(A^\theta)\rightarrow $ and two positive constants $C$ and $C_\theta$ such that
\begin{equation}
\label{pol-lyapunov-linear-equation}
\langle A z,P_\theta z\rangle_H + \langle P_\theta z, Az\rangle_H\leq-C\Vert z\Vert^2_H,\quad \forall z\in D(A),
\end{equation}
and
\begin{equation}\label{eq:est-pol}
\langle e^{tA}z,P_\theta e^{tA}z\rangle_H\leq \frac{C_\theta}{(1+t)^{2\gamma-1}}\Vert z\Vert^2_{D(A^\theta)}, \quad \forall z\in D(A^\theta),\quad t\geq 0.
\end{equation}
One can choose $P_\theta$ in the latter equation in the form 
\begin{equation}
\label{p-theta}
P_\theta := \int_0^\infty (e^{sA})^\star e^{sA}ds + \alpha I_{D(A^\theta)},
\end{equation}
with $\alpha>0$.
%, and in $D(A^{\theta+1})$ if $A^{\theta+1}$ is well-defined.  
%\begin{equation}
%\label{eq:est-pol}
%\Vert P_\theta z\Vert_{H}\leq C_\theta \Vert z\Vert_{D(A)},\quad \forall z\in D(A).
%\end{equation}
\end{proposition}
%The last part of the proposition is a consequence of Remark~\ref{rem:DA}.
\begin{proof}\textbf{ of Proposition \ref{lyapunov-2018}: }
The proof is divided into two parts: the first part handles the necessary condition of item (ii) (i.e., the $\Rightarrow$ part), while the second part focuses on the sufficient condition (i.e, the $\Leftarrow$ part). 

\textbf{ ($\Rightarrow$):}
We assume that the origin of \eqref{linear-system} is polynomially stable. This part of the proof is inspired by \cite{DATKO1970610}. %Indeed, the Lyapunov operator $P_\theta$ is \eqref{p-theta}.

Since the origin of \eqref{linear-system} is $\theta$-globally polynomially stable, then, for all $z\in D(A^\theta)$ and every $T>0$,
\begin{equation}
\begin{split}
\int_0^T \Vert e^{sA}z\Vert_H^2 ds \leq & \int_{0}^T \frac{C}{(1+s)^{2\gamma-1}}\Vert z\Vert^2_{D(A^\theta)}ds\leq C\Vert z\Vert^2_{D(A^\theta)}\int_{0}^\infty \frac{ds}{(1+s)^{2\gamma-1}}\\
\leq & \frac{C}{2\gamma - 1} \Vert z\Vert^2_{D(A^\theta)}.
\end{split}
\end{equation}
This implies that $\int_0^\infty \Vert e^{sA}z\Vert_H^2 ds$ is convergent and strictly positive as long as $z$ is different from $0$.  Moreover, for every $t\geq 0$, the operator %In addition, it is easy to see that there exists a positive value $C_\theta$ such that \eqref{eq:est-pol} holds.
\begin{equation}
Q_\theta(t) = \int_0^t (e^{sA})^\star e^{sA} ds 
\end{equation}
satisfies the following properties, for all $z_1,z_2$ in $D(A^\theta)$:
\begin{itemize}
\item[$(i)$] the function $t\mapsto \langle Q_\theta(t)z_1,z_2\rangle$ is well-defined;
\item[$(ii)$] if $t_1\leq t_2$, then $0\leq \langle Q_\theta(t_1) z,z\rangle_H\leq \langle Q_\theta(t_2)z,z\rangle_H$;
\item[$(iii)$] $\langle Q_\theta(t) z_1,z_2\rangle_H = \langle z_1,Q_\theta(t) z_2\rangle_H$. 
We only provide an argument for itemm $(i)$ since the two others are straigthforward. One has
\begin{align*}|\langle Q_\theta(t)z_1,z_2\rangle|^2 =  & \left|\int_0^t \langle e^{sA} z_1,e^{sA}z_2\rangle_H ds\right|^2\\
\leq & \left(\int_0^t \Vert e^{sA}z_1\Vert_H \Vert e^{sA}z_2\Vert_H ds\right)^2\\
\leq & \left(\int_0^t \Vert e^{sA}z_1\Vert^2_H ds\right)\left(\int_0^t \Vert e^{sA}z_2\Vert^2_H ds\right)\\
\leq & \left(\int_0^\infty \Vert e^{sA}z_1\Vert_H^2 ds\right)\left(\int_0^\infty \Vert e^{sA}z_2\Vert_H^2 ds\right).
\end{align*}
\end{itemize}
Therefore, by the principle of uniform boundedness, it follows that
\begin{equation}
\sup_{0\leq t\leq \infty}\Vert Q_\theta(t)\Vert_{\mathcal{L}(D(A^\theta),H)}<+\infty.
\end{equation}
Then, using items $(ii)$ and $(iii)$, it follows that there exists a self-adjoint and positive-definite operator $Q_\theta:\: D(A^\theta)\rightarrow H$ such that, for each $z\in D(A^\theta)$
\begin{equation}
\lim_{t\rightarrow +\infty} \Vert Q_\theta(t)z - Q_\theta z\Vert_H = 0.
\end{equation}
Define the function $V:D(A)\to\mathbb{R}_+$ by \begin{equation}
V(z) = \langle Q_\theta z, z\rangle_H=\int_0^\infty \Vert e^{As}z\Vert_H^2 ds.
\end{equation}
Clearly, one has, for every $t\geq 0$, that
\begin{equation}
\begin{split}
V(\etA z) &= \langle Q_\theta \etA z, \etA z\rangle_H=\int_0^\infty \Vert e^{A(t+s)}z\Vert_H^2 ds\\
&=\int_t^\infty \Vert e^{sA}z\Vert_H^2 ds\leq 
 \frac{C^2 \Vert z\Vert^2_{D(A)}}{(2\gamma(\theta)-1)(1+t)^{2\gamma(\theta)-1}}.\label{eq:Vt}
\end{split}
\end{equation}

 Since $z\in D(A^\theta)$, the derivative of $V$ with respect to $t$ exists and is given by
\begin{align*}
\frac{d}{dt}V(e^{tA}z) = &2\langle Q_\theta A e^{tA}z,e^{tA}z\rangle_H\\
= & 2\lim_{\tau \rightarrow +\infty} \int_0^\tau \langle Ae^{(t+s)A}z,e^{(t+s)A}z\rangle_Hds\\
= & \lim_{\tau \rightarrow +\infty} \int_0^\tau \frac{d}{ds}\Vert e^{(t+s)A}z\Vert^2_H ds\\
= & -\Vert e^{tA}z\Vert^2_H.
\end{align*}
In addition, since the origin is polynomially stable, the following holds
\begin{align*}
V(e^{tA}z)= &\langle Q_\theta e^{tA} z,\etA z\rangle_H
=  \int_0^\infty \Vert e^{(t+s)A}\Vert_H^2 d \\
= & \int_t^{\infty} \Vert e^{sA}z\Vert_H^2 ds
\leq \frac{C^2 \Vert z\Vert_{D(A^\theta)}}{(2\gamma-1)(1+t)^{2\gamma-1}},
\end{align*}
If one sets $P_\theta$ as 
\begin{equation}
P_\theta:=Q_\theta + \alpha I_{D(A^\theta)},
\end{equation}
with $\alpha>0$, one gets a self-adjoint, positive-definite and coercive operator. Note that, for every $z\in D(A^\theta)$, one has
$$
\langle P_\theta z,z\rangle_H=V(z)+\alpha \Vert z\Vert^2_H.
$$
Hence, \eqref{pol-lyapunov-linear-equation} is satisfied as well as
 \eqref{eq:est-pol} since $\frac{d}{dt}\Vert z\Vert^2_H\leq 0$. 
This concludes the proof of the first part.

\textbf{($\Leftarrow$):} 
We assume that there exists $P_\theta : D(A^\theta)\rightarrow D(A^\theta)$ such that \eqref{pol-lyapunov-linear-equation} and \eqref{eq:est-pol} hold. For $z\in D(A^\theta)$, set $V(z)=\langle P_\theta z,z\rangle_H$. Using \eqref{pol-lyapunov-linear-equation}, the derivative of $V$ along the the dynamics \eqref{linear-system} yields
\begin{equation}
\label{first-pol-lyap}
\frac{d}{dt} V(e^{tA}z)\leq -C\Vert z\Vert^2_H.
\end{equation}
Using \eqref{eq:est-pol}, one has
\begin{equation}
\lim_{t\rightarrow +\infty} V(\etA z)=0,\quad \forall z\in D(A^\theta).
\end{equation}
Then, integrating the latter equation between any non negative time $t$ and $\infty$, one has
\begin{equation}
\label{technique-yacine}
V(\etA z)\geq C\int_{t}^\infty \Vert e^{sA} z\Vert^2_H ds. 
\end{equation}
Recalling that $H$-norm of $(\etA)_{t\geq 0}$ is non increasing, the following holds
\begin{equation}
\Vert \etA z\Vert^2_H \leq \Vert e^{sA}\Vert_H^2,\quad \forall s\in \left[t/2,t\right],
\end{equation}
for $t>0$. Integrating the latter inequality between $\frac{t}{2}$ and $t$ yields
\begin{equation}
\frac{t}2\Vert \etA z\Vert^2_H \leq \int_{\frac{t}{2}}^{t} \Vert e^{sA}z\Vert^2_H.
\end{equation}
Noticing that $1+t\leq 2t$, for $t\geq 1$, then the following holds
\begin{equation}
\begin{split}
(1+t)\Vert \etA z\Vert_H^2 \leq &4\int_{\frac{t}{2}}^t \Vert e^{sA} z\Vert_H^2 ds
\leq  4\int_{\frac{t}{2}}^\infty \Vert e^{sA}z\Vert_H^2
\leq \frac{4}{C} V(e^{\frac{tA}{2}}z),
\end{split}
\end{equation}
where we have used \eqref{technique-yacine}. Using \eqref{eq:est-pol}, we obtain
\begin{equation}
\Vert \etA z\Vert^2_H \leq 4\frac{C_\theta}{C}\frac{1}{(1+t)^{2\gamma}} \Vert z\Vert^2_{D(A^\theta)},
\end{equation}
which concludes the proof of the polynomial stability of \eqref{linear-system} and that of Proposition \ref{lyapunov-2018}. 
\end{proof}

\section{Linear infinite-dimensional systems subject to a nonlinear damping}
\label{sec_main}

In this section, we discuss the notion of nonlinear damping function and state our main results, that is, roughly speaking, the following: when modifying a stabilizing linear feedback law with a nonlinear damping function, we characterize the asymptotic decays of corresponding trajectories.  

\subsection{Linear control system with collocated feedback law}\label{ss: FS}

Let $H$ and $U$ be real Hilbert spaces equipped with the scalar product $\langle \cdot,\cdot\rangle_H$ and $\langle \cdot, \cdot\rangle_U$, respectively. Let $A:\: D(A)\subset H\rightarrow H$ be a (possibly unbounded) linear operator whose domain $D(A)$ is dense in $H$. We suppose moreover that $A$ generates a strongly semigroup of contractions denoted by $(e^{tA})_{t\geq 0}$. We denote by $A^\star$ its adjoint. Finally, let $B$ be a bounded operator from $U$ to $H$ (i.e., $B\in\mathcal{L}(U,H)$) and let us denote by $B^\star$ its adjoint.

We consider the infinite-dimensional linear control system given by
\begin{equation}
\label{control-system}
\left\{
\begin{split}
&\frac{d}{dt} z = Az + Bu,\\
&z(0)=z_0,
\end{split}
\right.
\end{equation}
where %$A:D(A)\subset H\rightarrow H$ generates a strongly semigroup of contractions and 
$u$ denotes the control. In addition, we will choose the following collocated feedback law
\begin{equation}
\label{feedback-law}
u=-kB^\star z,
\end{equation}
where $k$ is a positive constant. 

The corresponding closed-loop system is then written as follows
\begin{equation}
\label{lclosed-system}
\left\{
\begin{split}
&\frac{d}{dt}z = (A-kBB^\star)z:=\tilde A z,\\
&z(0)=z_0.
\end{split}
\right.
\end{equation}
Since $B$ is a bounded operator, the domain of $\tilde A$ coincides with $D(A)$. Moreover, it is easy to see that $\tilde{A}$ generates a strongly continuous semigroup of contractions.

The asymptotic stability of the origin of \eqref{lclosed-system} has to be precised. We then assume that the origin of \eqref{lclosed-system} is either globally exponentially stable or globally polynomially stable. Both hypotheses are collected just below.

\begin{hypothesis}[Exponential stability]
\label{hypothesis-exp}
Assume that the origin of \eqref{lclosed-system} is globally exponentially stable.
\end{hypothesis}

\begin{example}
\label{example-wave}
Let us consider the following linear wave equation
\begin{equation}
\label{wave-equation-linear}
\left\{
\begin{split}
&z_{tt} = \Laplace z - a(x) z_t,\quad (t,x)\in \mathbb{R}_+\times \Omega,\\
&z(t,x)=0,\quad (t,x)\in \mathbb{R}_+\times \Gamma,\\
&z(0,x)=z_0(x),\: z_t(0,x)=z_1(x),
\end{split}
\right.
\end{equation}
where $\Omega\subset \mathbb{R}^n$ ($n\geq 1$) is a bounded connected domain with a smooth boundary $\Gamma:=\partial \Omega$. The damping localization function $a(\cdot)$ is smooth, nonnegative and there exists a positive constant $a_0$ such that $a(x)\geq a_0$ on a non empty open subset $ \omega$ of $\Omega$. In other words, the open subset $\omega$ is actually the set where the control acts. The feedback control is said to be \textit{globally} distributed if $\omega=\Omega$ and locally distributed if $\Omega\setminus \omega$ has a positive Lebesgue measure.

Equation \eqref{wave-equation-linear} can be rewritten as an abstract control system \eqref{lclosed-system} setting $H:=H^1_0(\Omega)\times L^2(\Omega)$, $U=L^2(\Omega)$ and
\begin{equation}
\begin{split}
A: D(A)\subset H&\rightarrow H\\
\begin{bmatrix}
z_1 & z_2
\end{bmatrix}^\top &\mapsto \begin{bmatrix}
z_2 & \Laplace z_1
\end{bmatrix},
\end{split}
\end{equation}
\begin{equation}
\begin{split}
B: U&\rightarrow H,\\
u &\mapsto \begin{bmatrix}
0 & \sqrt{a(x)}u
\end{bmatrix}^\top, 
\end{split}
\end{equation}
where $$D(A):=(H^2(\Omega)\cap H^1_0(0,1))\times H_0^1(\Omega).$$

The adjoint operators of $A$ and $B$ are, respectively
\begin{equation}
\begin{split}
A^\star: D(A)\subset H &\rightarrow H\\
\begin{bmatrix}
z_1 & z_2
\end{bmatrix}^\top &\mapsto -A\begin{bmatrix}
z_1 & z_2
\end{bmatrix}^\top,
\end{split}
\end{equation}
and
\begin{equation}
\begin{split}
B^\star: H&\rightarrow U,\\
\begin{bmatrix}
z_1 & z_2
\end{bmatrix}^\top &\mapsto \sqrt{a(x)} z_2.
\end{split}
\end{equation}
A straightforward computation, combined with some integrations by parts, shows that
\begin{equation*}
\langle Az,z\rangle_H + \langle z,Az\rangle_H \leq 0,\quad \forall z\in D(A).
\end{equation*}
Hence, applying L\"umer-Phillips's theorem, it follows that $A$ generates a strongly continuous semigroup of contractions. Moreover, using \cite[Theorem 2.1.]{enrike1990exponential}, \eqref{wave-equation-linear} is globally exponentially stable provided that 
%\begin{equation}
%\label{condition-zuazua}
%\omega \text{ is a neighbourhood of } \Gamma.
%\end{equation}
$\omega$ is a neighbourhood of $\Gamma$.
In particular, using Proposition \ref{lyapunov90}, there exists a Lyapunov operator $P\in\mathcal{L}(H)$ such that a Lyapunov inequality holds. Therefore, Hypothesis \ref{hypothesis-exp} holds for \eqref{wave-equation-linear}.  
\end{example}

%\begin{remark}
%The proof of this result, provided in \eqref{condition-zuazua}, relies on a compact-uniqueness strategy. The condition \eqref{condition-zuazua} has been improved later on in \cite{fu2007exact} and more precise results has been refined with some microanalysis techniques (see e.g., \cite{Lebeau1996} and \cite{le2017geometric}). 
%\end{remark} 
%
%\begin{remark}[The $1$D case]
%In the case where $n=1$, the result is way simpler. Indeed, the exponential stability holds if and only if the damping function $a$ is nonnegative on all the domain and positive in the subdomain $\omega$. %Note that in next section, we will provide refined result for the $1$D wave eqution.
%\end{remark}

\begin{hypothesis}[Polynomial stability]
\label{hypothesis-poly}
Assume that the origin of \eqref{lclosed-system} is $1$-globally polynomially stable with $\gamma>\frac{1}{2}$. 
\end{hypothesis}

\begin{example}
%{\color{red} L'exemple n'est pas bon a mon avis car on a $\gamma=1/2$. Il faut en trouver un autre avec $\theta=1$ ET $\gamma>1/2$. Sinon apres le theoreme 3 est vide.}

Consider again on the wave equation \eqref{wave-equation-linear}. Suppose that the condition such that the exponential is not satisfied. Assume moreover that $\Omega$ is a torus (i.e., the boundary conditions are uniformly equal). Then, under some regularity assumptions on the damping function $a$, that are collected in \cite[Theorem 2.6]{anantharaman2014sharp}, the trajectory is $1$-globally polynomially stable, with $\gamma>\frac{1}{2}$. 

\end{example}

\subsection{Nonlinear damping functions}
As it has been noticed at the beginning of the section, we want to study the asymptotic behavior of the origin of \eqref{control-system} with \eqref{feedback-law} modified by a nonlinearity, namely the \textit{nonlinear damping function}. We provide next the definition of the nonlinear damping function. 

\begin{definition}[Nonlinear damping functions on $S$]
\label{def-sat}
Let $S$ be a real Banach space equipped with the norm $\Vert \cdot \Vert_S$. Assume moreover that $(U,S)$ is a rigged Hilbert space\footnote{We refer the interested reader to \cite{de2005role} for more details on rigged Hilbert spaces}, i.e., $S$ is a dense subspace of $U$ and that the following inclusions hold
\begin{equation}
S\subseteq U\subseteq S^\prime.
\end{equation}
In particular, the duality pairing between $S$ and $S^\prime$ is compatible with the inner product on $U$, in the sense that
\begin{equation}
(u,v)_{S\times S^\prime}=\langle u,v\rangle_U,\quad \forall u\in S\subset U,\quad \forall v\in U=U^\prime\subset S^\prime.
\end{equation} A function $\sigma:U\rightarrow S$ is said to be a \emph{nonlinear damping function on $U$} if there exists positive constants $C_1$ and $C_2$ such that the following properties hold true.
\begin{itemize}
%\item[1.] For any $s\in U$, one has
%\begin{equation}\label{eq:bdd-sat}
%\max(\Vert \sigma(s)\Vert_U,\Vert \sigma(s)\Vert_S)\leq C_1;
%\end{equation}
%\item[2.] For any $s,\tilde{s}\in S$, one has
%\begin{equation}
%\langle \sigma(s)-\sigma(\tilde{s}),s-\tilde{s}\rangle_U\geq 0;
%\end{equation}
\item[1.] The function $\sigma$ is locally Lipschitz.
\item[2.] The function $\sigma$ is maximal monotone, that is: for all $s_1,s_2\in U$, $\sigma$ satisfies
\begin{equation}
\label{max-monotone}
\langle \sigma(s_1)-\sigma(s_2),s_1-s_2\rangle_U \geq 0.
\end{equation}
\item[3.] For any $s\in U$, one has
\begin{equation}\label{eq:inf-sat}
\Vert \sigma(s) - C_1s\Vert_{S^\prime}\leq C_2h(\Vert s\Vert_S)\langle \sigma(s),s\rangle_U,
\end{equation}
where $h:\mathbb{R}\to\mathbb{R}_+$ is a continuous and non decreasing function satisfying $h(0)>0$.
%Moreover, there exists a positive value $C_h$ such $h$ satisfies the following inequality  ({\color{red}Y: tu en as vraiment besoin??})
%\begin{equation}
%h(\Vert s\Vert_S)\leq C_h \Vert s\Vert_S;
%\end{equation}
%which is non decreasing in a neighborhood of $\infty$.
%\item[3.] For any $s\in U$, one has
%\begin{equation}
%\label{eq:monotone}
%\langle \sigma(s),s\rangle_U \geq 0.
%\end{equation}
\end{itemize}
\end{definition}
%Note that Item $2$ implies that for non zero $s\in U$, one has
%\begin{equation}
%\label{eq:monotone}
%\langle \sigma(s),s\rangle_U> 0.
%\end{equation}
%There exists also a stronger property that such a nonlinear function might satisfy, namely the maximal-monotony. %Here is a definition of such a property.
%\begin{definition}[Maximal monotone weak damping function]
%\label{def-max-mon}
%A weak damping function is said to be maximal monotone if it satisfies the following inequality, for all $s_1,s_2\in U$
%\begin{equation}
%\label{max-monotone}
%\langle \sigma(s_1)-\sigma(s_2),s_1-s_2\rangle_U \geq 0.
%\end{equation}
%\end{definition}
%We refer to \cite{showalter1997monobanach} for a more detailed discussion about maximal monotone operators. 

\begin{example}[Some examples of nonlinear damping functions]
%Here are examples of weak damping functions borrowed from \cite{tarbouriech2011book_saturating} or \cite{slemrod1989mcss}.
We provide two sets of examples depending on the fact that $S=U$ or not.
\begin{itemize}
\item[1.] Suppose that $S:=U$. The saturation studied in \cite{slemrod1989mcss}, \cite{lasiecka2002saturation} and \cite{mcpa2017siam} is defined as follows, for all $s\in U$, 
\begin{equation}
\label{sath}
\mathfrak{sat}_U(s):=\left\{
\begin{split}
&\frac{s}{\Vert s\Vert_U}s_0\text{ if }\Vert s\Vert \geq s_0,\\
&s\hspace{1.1cm}\text{      if } \Vert s\Vert\leq s_0
\end{split}
\right.
\end{equation}
where the positive constant $s_0$ is called the saturation level. This operator clearly satisfies Item 1. of Definition \ref{def-sat}. The fact that this operator is globally Lipschitz is proven in \cite{slemrod1989mcss}. Moreover, one verifies easily that this operator satisfies Item 3. of Definition \ref{def-sat}. In \cite{seidman2001note}, this operator is proved to be $m$-dissipative, which implies that it is maximal monotone.

\item[2.] Suppose that $S:=L^\infty(0,1)$ and $U=L^2(0,1)$. In this case, $S^\prime$ is the space of finitely additive measures\footnote{Let $\Sigma$ be an algebra of sets of a given set $\Omega$. A function $\lambda:\Sigma\rightarrow \bar{\RR}$ is said to be a finitely additive signed measure if: (i) $\lambda(\emptyset)=0$; (ii) given $K_1,K_2\in\Sigma$, disjoint subsets, $\lambda(K_1\cup K_2)=\lambda(K_1)+\lambda(K_2)$. The corresponding space, which is a Banach space, is endowed with the norm of total variation.}. It contains the space $L^1$, and  $S^\prime$ is continuously embedded in $L^1(0,1)$ via the operator $u\mapsto \int_0^1 udx$ (see \cite[Remark 7, Page 102]{brezis2010functional}). Therefore, it is clear  that $(U,S)$ is a rigged Hilbert space. Moreover, one can write via the latter embedding that
\begin{equation}
\label{embedding-linfty}
(u,v)_{S\times S^\prime} = (v,u)_{L^1(0,1)\times L^\infty(0,1)}=\langle u,v\rangle_{L^2(0,1)}.
\end{equation}
For this case, we give two examples: the first one is a saturation, while the second one is borrowed from \cite{MVC}.
\begin{itemize}
\item[(i)]\emph{($L^\infty$-saturation)} Standard saturation functions can be defined as follows:
\begin{equation}
\label{satLinfty}
\begin{split}
L^2(0,1) &\rightarrow L^\infty(0,1),\\
s &\mapsto \mathfrak{sat}_{L^\infty(0,1)}(s),
\end{split}
\end{equation}
where
$\mathfrak{sat}_{L^\infty(0,1)}(s)(\cdot) = \sat(s(\cdot))$, where $\sat:\mathbb{R}\to\mathbb{R}$
is a non decreasing, locally Lipschitz function verifying, for some positive constant $C$, that $|\sat(s)-Cs|\leq s\sat(s)$ for every $s\in\mathbb{R}$. For instance, $\arctan$, $\tanh$ and the standard saturation functions $\sigma_0(s)=\frac{s}{\max(1,\vert s\vert)}$ are saturation functions. In all these cases, the function $h$ appearing in \eqref{eq:inf-sat} can be taken equal to one. Note moreover that the saturations are uniformly bounded.
%is defined as follows
%\begin{equation}
%\sat(s)=\left\{
%\begin{split}
%& -s_0 \text{ if } s\leq -s_0\\
%& s \hspace{0.6cm}\text{ if } -s_0\leq s\leq s_0\\
%& s_0\hspace{0.4cm}\text{ if } s\geq s_0
%\end{split}
%\right.
%\end{equation} 
%It is also encompassed in Definition \ref{def-sat}. Indeed, there exists a  positive constant $C$ such that 
%\begin{description}
%\item[$\bullet$] $\sat(\cdot)$  is real-valued, defined on $\mathbb{R}$, increasing and globally Lipschitz, which implies that it satisfies Item 1 of Definition \ref{def-sat};
%\item[$\bullet$] $|\sat(s)-Cs|\leq s\sat(s)$. Then, it satisfies Item 2 of Definition \ref{def-sat} with $h$ the identity function;
%\item[$\bullet$] Function $s\in\mathbb{R}\mapsto \sat(s)\in\mathbb{R}$ is odd. In particular, this implies that it satisfies Item 3 of Definition \ref{def-sat}. 
%%\item[$(iii)$] $|v_1\large(\sat(v_1+v_2)-\sat(v_1)\large)|\leq |v_2|\min\large(\sat(|v_1|)+\sat(|v_2|)\large)$.
%\end{description}
%The proofs of these items are given in \cite{bible_khalil} and \cite{tarbouriech2011book_saturating}. These functions are proved also to satisfy Item 3. of Definition \ref{def-sat}. 
\item[(ii)] We have also the following nonlinear damping function, also called \emph{weak damping}, and borrowed from \cite[Theorem 2.]{MVC}
\begin{equation*}
\sigma(s) \leq c|s|^q, \forall s\in\mathbb{R}
\end{equation*}
with $c\geq 0$ and $q< 1$ and such that $\sigma(0)=0$ and $\sigma^\prime(0)\geq 0$. In this case, we have $h(|s|)=|s|^{q-1}$.
\end{itemize}
\end{itemize}
\end{example} 

%\item[(ii)] \emph{(Weak damping)} Any $C^1$ real-valued function $\sigma$ such that $\sigma'$ is bounded,  $\sigma'(0)>0$ and there exists $C_1,C_2$ positive, $\alpha\leq \beta<1\in \RR$ with 
%\begin{equation}
%C_1\vert s\vert^{1+\alpha}\leq s\sigma(s)\leq C_2\vert s\vert^{1+\beta}, \ s\in\RR,
%\end{equation}
%gives rise to a weak damping function verifying the items of Definition~\ref{def-sat} with $h(x)=\min(1,\vert s\Vert^{-\alpha})$.
%\end{itemize} 
%\end{itemize}
%\end{example}
%\begin{remark}\label{rem:MVC} 
%The weak damping function given in \cite{MVC} still falls into Definition~\ref{def-sat} by taking into account \cite[Remark 2, page 290]{MVC}.
%\end{remark}
Consider now the following nonlinear dynamics
\begin{equation}
\label{nlclosed-loop}
\left\{
\begin{split}
&\frac{d}{dt}z = A_\sigma(z),\\
&z(0) = z_0,
\end{split}
\right.
\end{equation}
where the nonlinear operator $A_\sigma$ is defined as follows
\begin{equation}\label{eq:NLop}
\begin{split}
A_\sigma: D(A_\sigma)\subset H&\rightarrow H\\
z&\mapsto Az-\sqrt{k}B\sigma(\sqrt{k}B^\star z).
\end{split}
\end{equation}
with $D(A_\sigma)$ the domain of $A_\sigma$. Since $B$ is bounded, one clearly has that $D(A_\sigma)=D(A)$. 

For the latter system, there exist many results related to its well-posedness and the asymptotic stability of its origin. The following theorem collects some of them.
\begin{theorem}[Well-posedness and global asymptotic stability]
\label{thm-wp-gas}
\begin{itemize}
%\item[(i)] Suppose that $\sigma$ is a nonlinear damping. Hence, for all $z_0\in  D(A_\sigma)$, there exists a unique mild solution to \eqref{nlclosed-loop}. Moreover, $A_\sigma$ generates a strongly continuous semigroup of contractions, which we denote by $(W_\sigma(t))_{t\geq 0}$. Hence, there exists a unique solution to \eqref{nlclosed-loop} and, for all $z_0\in D(A)$, the following function
%\begin{equation}
%\label{eq:decrease:non}
%t\mapsto \Vert W_\sigma(t)z_0\Vert_H,
%\end{equation}
%is nonincreasing. 

\item[(i)] Suppose that $\sigma$ is a nonlinear damping. Therefore, there exist a unique strong solution to \eqref{nlclosed-loop} for every initial condition $z_0\in D(A)$. Moreover, the following functions 
\begin{equation}
\label{eq:Wde}
t\mapsto \Vert W_\sigma(t)z_0\Vert_H,\: t\mapsto \Vert A_\sigma W_\sigma(t)z_0\Vert_H,
\end{equation}
are nonincreasing.
\item[(ii)] Supposing that all the assumptions of the latter item hold and assuming moreover that $D(A)$ is compactly embedded in $H$, then the origin of \eqref{nlclosed-loop} is globally asymptotically stable, i.e., for every $z_0\in D(A)$,
\begin{equation}
\lim_{t\rightarrow +\infty} \Vert W_\sigma(t) z_0\Vert_H^2 = 0.
\end{equation} 
\end{itemize}
\end{theorem}

%The proof of first item of this theorem is provided in \cite[Theorem 3.17 and Remark 3.14]{brezis1973operateurs}.

The proof of the first item is provided in \cite[Lemma 2.1, Part IV, page 165]{showalter1997monobanach}. The second item has been proved in the specific case of hyperbolic systems in \cite{dafermos1978asymptotic} (differentiable nonlinear damping) and \cite{haraux1985stabilization} (non differentiable nonlinear damping). The proof of this item relies on the use of the LaSalle's Invariance Principle. For the well-posedness and the global asymptotic stability of the closed-loop system \eqref{nlclosed-loop} in the case where $\sigma$ is a saturation, we refer the interested reader to \cite{seidman2001note}, \cite{slemrod1989mcss} or more recently to \cite{map2017mcss}. 

\begin{remark}
In some cases, it is not immediate to check whether $D(A)$ is compactly embedded in $H$. For instance, we know that this holds for hyperbolic systems \cite{haraux1985stabilization} or the linearized Korteweg-de Vries equation \cite{map2017mcss}. Note that the global asymptotic stability does not give any information on the decay rate of the trajectory of the systems. Here, we do not aim at just proving that the origin of \eqref{nlclosed-loop} is globally asymptotically stable, but rather at characterizing the decay rate of trajectories. 
\end{remark} 

%{\color{red} Y: je pense qu'il mettre un petit argument qui s'inspire precisement de Haraux.}

%{\color{blue} S: Le resultat n'est pas clair pour moi, a part pour les systemes hyperboliques. J'ai decide de mettre nos discussions en remarque pour bien appuyer sur le fait que notre but n'est pas de prouver la stabilite, mais bien de la caracteriser.}

\subsection{Global asymptotic stability results}
The remaining parts of the paper aim at characterizing precisely the stability properties of the origin of \eqref{nlclosed-loop}. Before stating our main results, let us provide some stability definitions. 
\begin{definition}[Semi-global exponential stability] 
The origin of \eqref{nlclosed-loop} is said to be \emph{semi-globally exponentially stable in $D(A)$} if, for any positive $r$ and any initial condition satisfying $\Vert z_0\Vert_{D(A)}\leq r$, there exist two positive constants $\mu:=\mu(r)$ and $K:=K(r)$ such that
\begin{equation}
\Vert W_{\sigma}(t)z_0\Vert_H\leq K e^{-\mu t}\Vert z_0\Vert_H, \quad \forall t\geq 0.
\end{equation}
\end{definition}
This definition is inspired by \cite{mcpa2017siam}, which focuses on a particular nonlinear damping function, namely the \textit{saturation}. It is well known that a linear finite-dimensional system subject to a saturated controller cannot be globally exponentially stabilized (see \cite{sussmann1991saturation}). The semi-global exponential stability written just below can be thought as a global exponential stability that is not uniform with respect to the initial condition (i.e., the constant $C$ and $\mu$ depend on the bound of the initial condition). Note also that \cite[Theorem 2]{MVC} corresponds exactly to a semi-global exponential stability result. 

A similar definition can be stated for the case of the polynomial stability.
\begin{definition}[Semi-global polynomial stability]
\label{semi-glob-poly}
The origin of \eqref{nlclosed-loop} is said to be \emph{semi-globally polynomially stable in $D(A)$} if there exists a positive constant $\gamma$ and if, for any positive $r$ and any initial condition satisfying $\Vert z_0\Vert_{D(A)}\leq r$, there exists a positive constant $C:=C(r)$ such that
\begin{equation}
\Vert W_{\sigma}(t)z_0\Vert_H\leq \frac{C}{(1+t)^{\gamma}}\Vert z_0\Vert_{D(A)}, \quad \forall t\geq 0.
\end{equation} 
\end{definition}
As for the semi-global exponential stability, the semi-global polynomial stability is a global polynomial stability which is not uniform with respect to the initial condition.

We are now in position to state the main results of our paper. The first one is based on Hypothesis \ref{hypothesis-exp}.
\begin{theorem}[Semi-global exponential stability]
\label{thm-abstract-lyap}
Consider that $\sigma$ in \eqref{nlclosed-loop} is a nonlinear damping function satisfying Item 1. and 3. of Definition \ref{def-sat}. Assume that Hypothesis \ref{hypothesis-exp} holds. Then, we have the following results 
\begin{itemize}
\item[1.] If $S=U$, there exists a strict and global Lyapunov function for \eqref{nlclosed-loop}. %{\color{red}
%Pas bien explique mais j'y arrive pas}
\item[2.] If $S\neq U$, assume that $\sigma$ is also maximal monotone (i.e., $\sigma$ satisfies Item 2. of Def \ref{def-sat}) and that the following inequality holds
\begin{equation}
\label{continuous-embedding-exp}
\Vert B^\star s\Vert_S \leq c_S \Vert s\Vert_{D(A)},\quad \forall s\in D(A).
\end{equation}
Hence, the origin of \eqref{nlclosed-loop} is semi-globally exponentially stable in $D(A)$.  
\end{itemize} 
\end{theorem}

A similar theorem can be stated when assuming that Hypothesis \ref{hypothesis-poly} holds.

\begin{theorem}[Semi-global polynomial stability]
\label{thm-poly}
Consider that $\sigma$ in \eqref{nlclosed-loop} is a nonlinear damping function satisfying all the items of Definition \ref{def-sat}. Assume moreover that $S=U$ and Hypothesis \ref{hypothesis-poly} holds. Then, the origin of \eqref{nlclosed-loop} is semi-globally polynomially stable with $\gamma=\frac{1}{2}$. 
\end{theorem}
\begin{remark}
The conclusion of Item $1.$ actually holds without the assumption of maximal monotonicity of the nonlinear damping function $\sigma$. Moreover, this case is entirely similar to the finite-dimensional one, cf. Appendix~\ref{app-finite} below.
\end{remark}
\begin{remark}[On the property \eqref{continuous-embedding-exp}]
In general, the property \eqref{continuous-embedding-exp} should be weaken, especially in the case of the wave equation in dimension higher than one. A more general assumption would be the following: there exists a positive $p\in\mathbb{N}$ and a positive constant $c_S$ such that
\begin{equation*}
\Vert B^\star s\Vert_S \leq c_S \Vert s\Vert_{D(A^p)}.
\end{equation*} 
However, in that case, applying directly our strategy fails, because it would require dissipativity of the semigroup in $D(A^p)$, which is in general false in dimension higher than $1$ and for $p\geq 2$.  
\end{remark}

\begin{remark}[On the polynomial stability result]
Theorem \ref{thm-poly} states a result only for the case where $S=U$. Indeed, the case $S\neq U$ would need a dissipativity property in $D(A^2)$, which is not true in general. 
\end{remark}

%\begin{remark}[The case $S\neq U$ do not work for ]
%
%\end{remark}

\section{Proof of the main theorems}
\label{sec_proof}

In this section, we provide the proof of Theorem \ref{thm-abstract-lyap} and the proof of Theorem \ref{thm-poly}. These proofs are based on a Lyapunov strategy.%, really similar to the one provided in \cite{liu1996finite}. 

\subsection{Proof of Theorem \ref{thm-abstract-lyap}}

%First, note that if we prove the attractivity of \eqref{nlclosed-loop} for strong solutions, the result follows for weak solutions by applying Lemma \ref{lemma-density-contraction}.

We split the proof of Theorem \ref{thm-abstract-lyap} into two cases. Firstly, we tackle Item 1. of Theorem \ref{thm-abstract-lyap} and then Item 2. Indeed, the Lyapunov functions considered in the two cases are different. In both cases, each argument is itself divided into two steps. First, we find a strict Lyapunov function
and then we prove the asymptotic stability of the origin of \eqref{nlclosed-loop}. 

\textbf{Case 1: $S=U$.}

Set $\tilde{A}=A-C_1BB^\star$, where the positive constant $C_1$ is given in \eqref{eq:inf-sat} and $P\in \mathcal{L}(H)$ is defined in \eqref{lyapunov-linear-equation}. Consider the following candidate Lyapunov function
\begin{equation}\label{eq:tildeV}
\tilde V(z):=\langle Pz,z\rangle_H+M\int_0^{\Vert z\Vert^2} \sqrt{v} h(\Vert B\Vert_{\mathcal{L}(U,H)} \sqrt{v})dv,
\end{equation}
where $M$ is a sufficiently large positive constant to be chosen later and $h$ is the function defined in Item 3 of Definition \ref{def-sat}. This function, inspired by \cite{liu1996finite}, is positive definite and coercive. Indeed, since $h(0)>0$, 
%there exists a positive value $C_h$ such that $h(0)\geq C_h$. Therefore, using the fact that $h$ is increasing, the following holds
%$$
%h(\Vert B\Vert_{\mathcal{L}(U,H)}\sqrt{v})\geq C_h,\quad \forall v\in [0,\Vert z\Vert_H^2 ].
%$$
%The latter implies
\begin{equation}
\int_{0}^{\Vert z\Vert^2_H} \sqrt{v} h(\Vert B\Vert_{\mathcal{L}(U,H)}\sqrt{v}) dv\geq h(0) \int_{0}^{\Vert z\Vert^2_H} \sqrt{v} dv = \frac{2h(0)}{3} \Vert z\Vert^3_H.
\end{equation}
Noticing that there exists $\alpha>0$ such that $\alpha\Vert z\Vert^2\leq \langle Pz,z\rangle_H\leq \Vert P\Vert_{\mathcal{L}(H)} \Vert z\Vert^2$, one has 
\begin{equation}
\label{Lyapunov-semi-bounds}
\alpha\Vert z\Vert^2+Mh(0) \frac{2}{3} \Vert z\Vert^3_H \leq \tilde{V}\leq \Vert P\Vert_{\mathcal{L}(H)} \Vert z\Vert^2 + M\Vert z\Vert_H^3 h(\Vert B\Vert_{\mathcal{L}(U,H)}\Vert z\Vert_H).
\end{equation}

Applying Cauchy-Schwarz's inequality, one has
\begin{align} 
\frac{d}{dt}\langle Pz,z\rangle_H= &\langle Pz,\tilde{A}z\rangle_H+\langle P\tilde{A}z,z\rangle_H \\
&+\langle Pz,B(C_1B^\star z-\sigma(B^\star z)\rangle_H+\langle PB(C_1B^\star z-\sigma(B^\star z),z\rangle_H\nonumber\\
\leq & -C\Vert z\Vert^2_H+2\langle  B^\star Pz,C_1B^\star z-\sigma(B^\star z)\rangle_U\label{eq:SU}\\
\leq &-C\Vert z\Vert^2_H+2\Vert B^\star Pz\Vert_U\Vert C_1B^\star z-\sigma(B^\star z)\Vert_U,\nonumber
\end{align}
%for some positive constant $C$.
Using Item 3. of Definition \ref{def-sat} and the fact that $B^\star$ is bounded in $U$, it yields
\begin{align*}
\frac{d}{dt}\langle Pz,z\rangle_H\leq &- C\Vert z\Vert_H^2 + 2C_2\Vert B^\star\Vert_{\mathcal{L}(H,U)}\Vert P\Vert_{\mathcal{L}(H)}\Vert z\Vert_{H} h(\Vert B^\star z\Vert_U)
\langle B^\star z,\sigma(B^\star z)\rangle_U\\
\leq & -C\Vert z\Vert_H^2 + 2C_2\Vert B^\star\Vert_{\mathcal{L}(H,U)}\Vert P\Vert_{\mathcal{L}(H)}\Vert z\Vert_{H} h(\Vert B\Vert_{\mathcal{L}(U,H)} \Vert z\Vert_U)
\langle B^\star z,\sigma(B^\star z)\rangle_U.
\end{align*}
Secondly, using the dissipativity of the operator $A$, one has
\begin{align*}
M\frac{d}{dt} \int_0^{\Vert z\Vert^2}  \sqrt{v} h(\Vert B\Vert_{\mathcal{L}(U,H)}\sqrt{v})dv = &M\Vert z\Vert_H h(\Vert B\Vert_{\mathcal{L}(U,H)} \Vert z\Vert_H)(\langle A z,z\rangle_H\\
& + \langle z,Az\rangle - 2\langle B^\star z,\sigma(B^\star z)\rangle_U)\\
\leq & -2M\Vert z\Vert_H h(\Vert B\Vert_{\mathcal{L}(U,H)} \Vert z\Vert_H)\langle B^\star z,\sigma(B^\star z)\rangle_U
\end{align*}
Hence, if one chooses $M$ as 
\begin{equation}
M = C_2 \Vert B^\star\Vert_{\mathcal{L}(H,U)}\Vert P\Vert_{\mathcal{L}(H)},
\end{equation}
one obtains, after adding the above two equations, that
\begin{equation}
\label{abs-strict-lyap}
\frac{d}{dt} \tilde{V}(z) \leq - C\Vert z\Vert^2_H.
\end{equation}
This concludes the proof of Theorem \ref{thm-abstract-lyap} in the case where $S=U$.

\textbf{Case 2: $S\neq U$.}

In this case, we are not able to control the term $\Vert C_1B^\star z-\sigma(B^\star z)\Vert_U$ in \eqref{eq:SU} with Item 3. of Definition \ref{def-sat}. To tackle this term, the inequality \eqref{continuous-embedding-exp} together with Item 1. of Theorem \ref{thm-wp-gas} will be used in order to prove that the origin of \eqref{nlclosed-loop} is semi-globally exponentially stable. %Using Lemma \ref{semiexp-to-globas}, {\color{red} Y: je ne trouve pas ce lemme} it proves that the origin of \eqref{nlclosed-loop} is globally asymptotically stable. 

Let $\tilde{V}(z)$ be the Lyapunov function candidate defined by
\begin{equation}
z\in D(A)\mapsto \tilde{V}(z):=\langle Pz,z\rangle_H + M\Vert z\Vert^2_H, 
\end{equation}
where $M>0$ will be selected later.

First, using the dissipativity of the operator $A$, one has
\begin{equation}
\frac{d}{dt}M\Vert z\Vert^2_H\leq -2M\langle B^\star z,\sigma(B^\star z)\rangle_U.
\end{equation}

Second, perfoming similar computations than in the case $S=U$, one obtains 
\begin{equation}
\frac{d}{dt}\langle Pz,z\rangle_H \leq  -C\Vert z\Vert^2_H+2\langle  B^\star Pz,C_1B^\star z-\sigma(B^\star z)\rangle_U.
\end{equation}  
It remains now to control the term
$$
2\langle  B^\star Pz,C_1B^\star z-\sigma(B^\star z)\rangle_U.
$$
We now assume that we have a strong solution for \eqref{nlclosed-loop}, whose initial condition $z_0\in D(A)$ is such that
\begin{equation}
\Vert z_0\Vert_{D(A)} \leq r,\: \Vert z_0\Vert_H\leq r,
\end{equation} 
for some positive $r$.
Since $(U,S)$ is a rigged Hilbert space, hence the following holds 
$$
\langle  B^\star Pz,C_1B^\star z-\sigma(B^\star z)\rangle_U=
(B^\star P z,C_1B^\star z-\sigma(B^\star z))_{S\times S^\prime}.
$$
Hence, applying Cauchy-Schwarz's inequality, one obtains
\begin{align}
\frac{d}{dt}\langle Pz,z\rangle_H \leq -C\Vert z\Vert^2_H + 2\Vert B^\star P z\Vert_S\Vert C_1B^\star z-\sigma(B^\star z)\Vert_{S^\prime}.
\end{align}
Moreover, thanks to \eqref{continuous-embedding-exp}, one has
\begin{equation}
\label{eq:crux1}
\Vert B^\star P z\Vert_S \leq c_S\Vert P z\Vert_{D(A)}
\end{equation}
and
\begin{equation}
\begin{split}
\label{eq:crux2}
\Vert C_1B^\star z-\sigma(B^\star z)\Vert_{S^\prime}\leq &C_2h(\Vert B^\star z\Vert_S)\langle B^\star z,\sigma(B^\star z)\rangle_U\\
\leq  &C_2 h(\Vert B^\star\Vert \Vert z\Vert_{D(A)})\langle B^\star z,\sigma(B^\star z)\rangle_U,
\end{split}
\end{equation}
where we have used  the fact that $h$ is non decreasing  and Item 3. of Definition \ref{def-sat} in the second one. 

%\begin{equation}
%\begin{align}
%\frac{d}{dt}\langle Pz,z\rangle_H \leq &- C\Vert z\Vert^2_H + 2c_SC_2\Vert P z\Vert_{D(A)}h(\Vert B^\star z\Vert_S)\langle B^\star z,\sigma(B^\star z)\rangle_U\nonumber\\
%\leq & -C\Vert z\Vert_H^2 + 2c_S C_2\Vert Pz\Vert_{D(A)}h(\Vert z\Vert_{D(A)})\langle B^\star z,\sigma(B^\star z)\rangle_U.
%\end{align}
%\end{equation}
%Noticing that
%\begin{equation}
%\Vert P z\Vert_{D(A)} = \Vert Pz\Vert_H + \Vert AP z\Vert \leq \Vert P\Vert_{\mathcal{L}(H)} \Vert z\Vert_{D(A)}
%\end{equation}
%where, in the second line, it has been used the assumption given in \eqref{continuous-embedding}.

Now, using \eqref{eq:Wde}, the fact that $P\in\mathcal{L}(D(A))$ and the dissipativity of the strong solution, which comes from Item 2. of Theorem \ref{thm-wp-gas} and which can be written as follows:
\begin{equation}\label{eq-DA}
\Vert PW_{\sigma}(t)z_0\Vert_{D(A)}\leq \Vert P\Vert_{\mathcal{L}(D(A))}\Vert z_0\Vert_{D(A)}, \: \Vert W_\sigma(t) z_0\Vert_{D(A)} \leq \Vert z_0\Vert_{D(A)},
\end{equation}
one has
\begin{equation}\label{eq:good0}
\frac{d}{dt}\langle Pz,z\rangle_H \leq -C\Vert z\Vert^2_H + 2c_S h(\Vert B^\star\Vert r)rC_2\Vert P\Vert_{\mathcal{L}(D(A)}\langle B^\star z,\sigma(B^\star z)\rangle_U.
\end{equation}
Therefore, if one selects $M$ such that 
\begin{equation}
M= c_S C_2 h(\Vert B^\star\Vert r)r\Vert P\Vert_{\mathcal{L}(D(A))},
\end{equation} 
it follows
\begin{equation}
\frac{d}{dt}\tilde{V}(z)\leq -C\Vert z\Vert^2_H.
\end{equation}
Note that we have, for all $z\in H$
\begin{equation}
\langle Pz,z\rangle_H \leq \Vert P\Vert_{\mathcal{L}(H)}\Vert z\Vert_H.
\end{equation}
Hence, it yields
\begin{equation}
\begin{split}
\frac{d}{dt}\tilde{V}(z) \leq &- \frac{C}{2\Vert P\Vert_{\mathcal{L}(H)}}\langle Pz,z\rangle_H - \frac{C}{2}\Vert z\Vert^2_H\\
\leq &-\mu \tilde{V}(z),
\end{split}
\end{equation}
where
\begin{equation}
\mu:=\min\left(\frac{C}{2\Vert P\Vert_{\mathcal{L}(H)}},\frac{C}{2M}\right).
\end{equation}
%Therefore, there exists a positive value $\mu:=\mu\left(r,\Vert P\Vert_{\mathcal{L}(H)},C\right)$ such that
%\begin{equation}
%\frac{d}{dt}\tilde{V}(z) \leq -\mu \tilde{V}(z).
%\end{equation}
After integration of the above differential inequality, one obtains
\begin{equation}
\tilde{V}(W_{\sigma}(t)z_0) \leq e^{-\mu t}\tilde{V}(z_0),\quad \forall t\geq 0. 
\end{equation}
Hence,
\begin{equation}
\Vert W_{\sigma}(t)z_0\Vert_H^2 \leq \frac{\Vert P\Vert_{\mathcal{L}(H)}+M}{M}e^{-\mu t}\Vert z_0\Vert^2_{H}. 
\end{equation}
Since $M$ depends on the bound of the initial condition, the origin of \eqref{nlclosed-loop} is semi-globally exponentially stable for any strong solution to \eqref{nlclosed-loop}. It concludes the proof of Item 2. of Theorem \ref{thm-abstract-lyap}.  %It yields that the origin is globally asymptotically stable from Lemma \ref{semiexp-to-globas}. Moreover, using Lemma \ref{lemma-density-contraction}, the origin of \eqref{nlclosed-loop} is globally asymptotically for any weak solution to \eqref{nlclosed-loop}. It concludes the proof of Item 2. of Theorem \ref{thm-abstract-lyap}. 
%\begin{remark}\label{S=U} 
%In the case $S=U$, we actually proceed exactly as in the finite-dimensional case (see Appendix) and we actually have a {\emph global} Lyapunov function independent of the initial condition which provides an optimal time of decrease.
%\end{remark}
%\begin{remark}\label{SnotU}
%Bounding $\Vert B^\star P z\Vert_S$ and $\Vert B^\star z\Vert_S$ to derive \eqref{eq:good0} required to have estimates in $D(A)$, thanks to \eqref{continuous-embedding}. In turn, to use the latter equation, one had to rely crucially on the maximal monotonicity of $\sigma$. This is where the main obstacle to extend 
%the results of finite-dimension or $S=U$ to the more interesting case $S\neq U$ lies.  On the other hand, \cite{MVC} present very interesting results where $\sigma$ is not assumed to be maximal monotone. We present in the last section some partial results for getting bounds directly in $S$ and without relying on estimates in $D(A)$.
%\end{remark}

\begin{remark}
The Lyapunov functional used in the proof of the case $S=U$ corresponds to the one used in the finite-dimensional case, treated in Appendix \ref{app-finite}. In particular, one can characterize the asymptotic behavior of the trajectory in a similar manner than in Remark \ref{rem-asymptotic-behavior} by setting $\lambda_{\min}(P):=\alpha$ and $K: X\in\mathbb{R}_+\mapsto \int_0^{X} \sqrt{v} h(\Vert B\Vert_{\mathcal{L}(U,H)}\sqrt{v})dv\in\mathbb{R}_+$. This implies in particular that we do not need the solution to be strong, which is in contrast with the case where $S\neq U$, where the decay rate depends on the bound of the initial condition in $D(A)$.
\end{remark}

\subsection{Proof of Theorem~\ref{thm-poly}}

We assume here that $S=U$ and $\theta=1$. Set $\tilde{A}=A-C_1BB^\star$, $P_1: D(A)\rightarrow D(A)$ defined in \eqref{pol-lyapunov-linear-equation} and $C_1$ is the positive constant defined in Item 2 of Definition \ref{def-sat}.

Let us consider the following candidate Lyapunov function
\begin{equation}
\tilde{V}(z) = \langle P_1 z,z\rangle_H + M\Vert z\Vert_H^2,
\end{equation}
where $M$ is a positive constant that has to be chosen.

First, using the dissipativity of the operator $A$, one has
\begin{equation}
\label{compensate}
\begin{split}
M \frac{d}{dt} \Vert z\Vert^2_H = &M(\langle Az,z\rangle_H + \langle  z, Az\rangle_H + 2\langle \sigma(B^\star z),B^\star z \rangle_U)
\leq - 2M\langle \sigma(B^\star z),B^\star z\rangle_U.
\end{split}
\end{equation}

Secondly, we have
\begin{equation*}
\begin{split}
\frac{d}{dt} \langle P_1 z,z\rangle_H = & \langle P_1 A_\sigma(z),z\rangle_H + \langle P_1 z,A_\sigma(z)\rangle_H
= \langle P_1 \tilde{A} z,z\rangle_H + \langle P_1 z,\tilde{A} z\rangle_H \\
& + 2 \langle P_1 B (C_1 B^\star z-\sigma(B^\star z),z\rangle_H
\leq  -C\Vert z\Vert^2_H \\
&+ 2 \langle C_1 B^\star z - \sigma(B^\star z),B^\star P_1 z\rangle_{U},
\end{split}
\end{equation*}
where we have used in the last line the Lyapunov inequality \eqref{pol-lyapunov-linear-equation}. %and the fact that $(S,U)$ is a rigged Hilbert space.

Applying Cauchy-Schwarz inequality and Item 3. of Definition \ref{def-sat}, one obtains
\begin{equation}
\begin{split}
\frac{d}{dt} \langle P_1 z,z\rangle_H \leq &-C\Vert z\Vert^2_H + 2 \Vert C_1 B^\star z - \sigma(B^\star z)\Vert_U\Vert B^\star P_1 z\Vert_U\\
\leq & - C\Vert z\Vert^2_H + 2 C_2 h(\Vert B^\star z\Vert_S) \langle \sigma(B^\star z),B^\star z\rangle_U \Vert B^\star P_1 z\Vert_{U}
\end{split} 
\end{equation}

It remains to choose a constant $M$ in \eqref{compensate} in order to compensate the term
\begin{equation}
\label{eq:nl-poly}
 C_2 h(\Vert B^\star z\Vert_U)\Vert B^\star P_1 z\Vert_{U},
\end{equation}
which appears in the latter inequality.
%In order to so, we split now the proof into two cases: either $S=U$ or $S\neq U$. 

%\paragraph{First case: $S=U$.}  

We consider initial condition $z_0$ in $D(A)$ satisfying
\begin{equation}
\label{bounded-IC-poly}
\Vert z_0\Vert_{D(A)} \leq r,
\end{equation}
for some positive $r$.
Note that, since $\sigma$ is a nonlinear damping, one can apply Item 1. of Theorem \ref{thm-wp-gas}. Therefore, for all $z_0\in D(A)$,
\begin{equation}
\Vert W_\sigma(t) z_0\Vert_H \leq \Vert z_0\Vert_H,\: \Vert A W_\sigma(t) z_0\Vert_H \leq \Vert Az_0\Vert_H.
\end{equation}
This latter property together with \eqref{bounded-IC-poly} implies that
\begin{equation}
\label{dissipativity-poly-nl}
\Vert W_\sigma(t) z_0\Vert_{D(A)}\leq r.
\end{equation}
Finally, using the fact that $B^\star \in \mathcal{L}(H,U)$ and that $h$ is non decreasing, \eqref{eq:nl-poly} becomes
\begin{equation}
\begin{split}
C_2 h(\Vert B^\star z\Vert_U)\Vert B^\star P_1 z\Vert_{U}\leq & C_2 h(\Vert B^\star\Vert_{\mathcal{L}(H,U)} \Vert z\Vert_{H}) \Vert B^\star\Vert_{\mathcal{L}(H,U)}   \Vert P_1 z\Vert_H\\
\leq & C_2 C_\theta h(\Vert B^\star\Vert_{\mathcal{L}(H,U)} \Vert z\Vert_{H}) \Vert B^\star\Vert_{\mathcal{L}(H,U)} \Vert z\Vert_{D(A)}.
\end{split}
\end{equation}
Then, using \eqref{dissipativity-poly-nl} and the fact that $\Vert z\Vert_H\leq \Vert z\Vert_{D(A)}$, one has
\begin{equation}
C_2 h(\Vert B^\star z\Vert_U)\Vert B^\star P_1 z\Vert_{U} \leq C_2C_\theta h(\Vert B^\star\Vert_{\mathcal{L}(H,U)} r) \Vert B^\star\Vert_{\mathcal{L}(H,U)} r.
\end{equation} 
Finally, if one selects $M$ such that
\begin{equation}
M= C_2C_\theta h(\Vert B^\star\Vert_{\mathcal{L}(H,U)}r) \Vert B^\star\Vert_{\mathcal{L}(H,U)} r,
\end{equation}
the derivative of $\tilde{V}$ along the trajectories of \eqref{nlclosed-loop} satisfies
\begin{equation}
\label{ineq-poly-nl}
\tilde{V}(z)\leq - C\Vert z\Vert^2_H,\quad \forall z\in D(A).
\end{equation}
Note that $\tilde{V}$ satisfies, for all $z\in D(A)$
\begin{equation}
\label{V-bounded-poly}
\alpha \Vert z\Vert^2_{D(A)} +M\Vert z\Vert_H^2 \leq \tilde{V}(z)\leq M\Vert z\Vert_H^2 + C_\theta \Vert z\Vert^2_{D(A)}.
\end{equation}
First, integrating \eqref{ineq-poly-nl} between $0$ and $t$, one obtains
\begin{equation}\label{eq:W}
\tilde{V}(W_\sigma(t)z_0)-\tilde{V}(z_0)\leq -C\int_0^t \Vert W_\sigma(s) z_0 \Vert^2_H ds.
\end{equation}
%Therefore, using \eqref{V-bounded-poly} and the facts that $\Vert z_0\Vert_H\leq \Vert z_0\Vert_{D(A)}$ and $t\mapsto \Vert W_\sigma(t) z_0 \Vert^2_H$ is decreasing, one has
%\begin{equation}
%(\alpha +M+Ct)\Vert W_\sigma(t) z_0\Vert^2_H\leq \tilde{V}(z_0)\leq (M+C_\theta) \Vert z_0\Vert_{D(A)}.
%\end{equation}
%One deduces that, for all $t\geq 0$,
Since $A_\sigma$ is dissipative, one deduces from \eqref{eq:W} that, for every $t\geq 0$,
$$
C(1+t) \Vert W_\sigma(t) z_0\Vert_H^2\leq C\Vert z_0\Vert_H^2+\tilde{V}(z_0),
$$
and hence, 
%Second, noticing that $\frac{d}{dt}\Vert z\Vert_H^2 =\langle A_\sigma(z),z\rangle_H + \langle z,A_\sigma(z)\rangle_H\leq 0$, one has
%\begin{align*}
%\begin{split}
%\frac{d}{dt}\left\{(1+t)\Vert z\Vert_H^2\right\}=&\Vert z\Vert^2_H + (1+t)\frac{d}{dt}\Vert z\Vert^2_H\\
%\leq &\Vert z\Vert^2_H
%\end{split}
%\end{align*}
%Integrating the latter inequality between $0$ and $t$, one obtains
%\begin{equation*}
%\begin{split}
%(1+t)\Vert W_\sigma(t) z_0\Vert_H^2 - \Vert z_0\Vert^2_H = &\int_0^t \Vert W_\sigma(s) z_0\Vert^2_H ds
%\leq \frac{M+C_\theta + C}{C} \Vert z_0\Vert_H^2. 
%\end{split}
%\end{equation*}
\begin{equation}
\Vert W_\sigma(t) z_0\Vert_H^2 \leq \frac{1}{1+t}\frac{M+C_\theta +C}{C} \Vert z_0\Vert_H^2.
\end{equation}
This achieves the proof of Theorem \ref{thm-poly}.

\section{Illustrative examples}
\label{sec_example}

\subsection{Linearized Korteweg-de Vries equation with spatially localized damping}

As a first example, let us focus on the following partial differential equation,
\begin{equation}
\left\{
\begin{split}
&z_t(t,x) + z_x(t,x) + z_{xxx}(t,x)=-a(x)z(t,x),\quad (t,x)\in \mathbb{R}_{\geq 0}\times [0,L],\\
&z(t,0) = z(t,L)=z_x(t,L)=0,\quad t\in \mathbb{R}_{\geq 0},\\
&z(0,x)=z_0(x),\quad x\in [0,L],
\end{split}
\right.
\end{equation}
where $L$ is a positive constant, $\omega$ is a nonempty open subset of $(0,L)$ and $a(x)$ is a smooth bounded nonnegative function satisfying $a(x)\geq a_0$ for all $x\in\omega$ for some positive constant $a_0$.

This equation can be written in an abstract way as in \eqref{lclosed-system} if one sets $H=L^2(0,L)$, $U=L^2(\Omega)$,
\begin{equation}
\begin{split} 
A:D(A)\subset L^2(0,L)&\rightarrow L^2(0,L),\\
z&\mapsto -z^\prime-z^{\prime\prime\prime},
\end{split}
\end{equation}
where 
\begin{equation}
D(A) :=\lbrace z\in H^3(0,L)\mid z(0)=z(L)=z^\prime(L)=0\rbrace,
\end{equation}
and 
\begin{equation}
\begin{split}
B : L^2(\omega) \rightarrow L^2(\Omega)\\
 u \mapsto \sqrt{a(x)} u.
\end{split}
\end{equation}\footnote{We refer to \cite{haraux1985stabilization} for this definition in the case of hyperbolic system.}
The adjoint operators of $A$ and $B$ are, respectively
\begin{equation}
\begin{split}
A^\star: D(A^\star)\subset H&\rightarrow H,\\
z & \mapsto z^\prime + z^{\prime\prime\prime},
\end{split}
\end{equation}
with $D(A^\star):= \lbrace z\in H^3(0,L)\mid z(0)=z(L)=z^\prime(0)=0\rbrace$, and
\begin{equation}
\begin{split}
B^\star : L^2(\Omega)&\rightarrow L^2(\omega)\\
z &\mapsto \sqrt{a(x)}z.
\end{split}
\end{equation}
A straightforward computation, together with some integrations by parts, shows that
\begin{equation}
\langle Az,z\rangle_H + \langle z,Az\rangle_H\leq 0,\quad \forall z\in D(A). 
\end{equation}
Since $A$ is a closed linear operator and $D(A)$ is dense in $H$, according to L\"umer-Phillips' theorem (see e.g., \cite[Theorem 3.8.4., Page 103]{tucsnak2009observation}), it follows that $A$ generates a strongly continuous semi-group of contractions. Note that, according to \cite[Section 4]{cerpa2013control}, Hypothesis \ref{hypothesis-exp} holds.

In the case where $S=U$, the result follows easily, since the operator $B^\star$ does not any regularity propert as in the case where $S\neq U$. Consider now the saturation $\sigma$ defined in \eqref{satLinfty}, i.e. $\sigma=\satl$. In order to check whether \eqref{embedding-linfty} holds, the following result, obtained in \cite{mcpa2015kdv_saturating}, is needed:  %Note that, in the case of the Korteweg-de Vries equation, the Lyapunov operator is as follows $P:=I_H$.

\begin{lemma}[\cite{mcpa2015kdv_saturating}, Lemma 4.]
For all $z\in D(A)$, there exists a positive constant $\Delta$ such that
\begin{equation*}
\Vert z\Vert_{H^1_0(0,L)}\leq \Delta \Vert z\Vert_{D(A)}.
\end{equation*}
\end{lemma}
Using the above mentionned result together with the fact that the space $H^1_0(0,L)$ is continuously embedded in $L^\infty(0,L)$, that is due to Rellich-Kondaroch Theorem \cite[Theorem 9.16, page 285]{brezis2010functional}, one obtains that
\begin{equation*}
\Vert z\Vert_S\leq \Delta \Vert z\Vert_{D(A)}.
\end{equation*}
Since $B^\star$ is bounded in $L^\infty$, there exists a positive constant $c_B$ such that $\Vert B^\star z\Vert_S \leq c_B\Vert z\Vert_{S}$, and then
\begin{equation*}
\Vert B^\star z\Vert_{S}\leq \Delta\Vert z\Vert_{D(A)}.
\end{equation*}
Therefore, \eqref{continuous-embedding-exp} holds for the linear Korteweg-de Vries equation.

Finally, since all the properties needed to apply Theorem \ref{thm-abstract-lyap} hold, this proves that the origin of
\begin{equation}
\label{kdv-without}
\left\{
\begin{split}
&z_t(t,x) + z_x(t,x) + z_{xxx}(t,x)=- \sqrt{a(x)}\mathfrak{sat}_{L^\infty}(\sqrt{a(x)}z(t,x)),\quad (t,x)\in \mathbb{R}_{\geq 0}\times [0,L],\\
&z(t,0) = z(t,L)=z_x(t,L)=0,\quad t\in \mathbb{R}_{\geq 0},\\
&z(0,x)=z_0(x),\quad x\in [0,L],
\end{split}
\right.
\end{equation}
is semi-globally exponentially stable in $D(A)$. 

%We perform now some numerical simulations. Our numerical scheme is inspired by \cite{nm_KdV}.

%\begin{remark}
%This equation has been studied in more details in \cite{mcpa2017siam}. To be more precise, the latter article studies a nonlinear version of the Korteweg-de Vries equation, that is
%\begin{equation}
%\left\{
%\begin{split}
%&z_t(t,x) + z_x(t,x) + z_{xxx}(t,x) + z(t,x)z_x(t,x)=- u(t,x),\quad (t,x)\in \mathbb{R}_{\geq 0}\times [0,L],\\
%&z(t,0) = z(t,L)=z_x(t,L)=0,\quad t\in \mathbb{R}_{\geq 0},\\
%&z(0,x)=z_0(x),\quad x\in [0,L],
%\end{split}
%\right.
%\end{equation}
%This latter KdV equation is studied with the nonlinear damping function given in \eqref{satLinfty}. A compactness-uniqueness strategy, inspired by \cite{rosier2006global}, is followed to prove the global asymptotic stability of the nonlinear PDE subject to a saturation. %Even if we do not consider the nonlinear version of the Korteweg-de Vries equation, the approach presented in the present paper allows us to conclude on the global asymptotic stability only by checking whether \eqref{continuous-embedding} holds. Moreover, we can also consider other kinds of nonlinearities.  
%\end{remark}

\subsection{Wave equation}

Consider Example \ref{wave-equation-linear} in the case where the damping is modified by a nonlinear damping function satisfying all the items of Definition \ref{def-sat}. Then, the equation reads as follows

\begin{equation}
\label{wave-equation-nonlinear}
\left\{
\begin{split}
&z_{tt} = \Laplace z - \sqrt{a(x)} \sigma(\sqrt{a(x)}z_t),\quad (t,x)\in \mathbb{R}_+\times \Omega,\\
&z(t,x)=0,\quad (t,x)\in \mathbb{R}_+\times \Gamma,\\
&z(0,x)=z_0(x),\: z_t(0,x)=z_1(x),
\end{split}
\right.
\end{equation}

As before, the case $S=U$ follows easily. Therefore, assuming that Hypothesis \ref{hypothesis-exp} holds, there exists a strict and global Lyapunov function for \eqref{wave-equation-nonlinear} and, assuming that Hypothesis \ref{hypothesis-poly} holds, the origin of \eqref{wave-equation-nonlinear} is semi-globally polynomially stable.

Assume now that $S=L^\infty(\Omega)$. Note that the inequality given by \eqref{continuous-embedding-exp} does not hold if the dimension of $x$ is higher or equal to $2$. Then, assume that $\Omega:=[0,1]$. This implies that
\begin{equation}
\begin{split}
\Vert B^\star z\Vert_S=&\Vert a(\cdot)z_t(t,\cdot)\Vert_{L^\infty(\Omega)}\\
\leq &\Vert a(\cdot)\Vert_{L^\infty(\Omega)}\Vert z_t(t,\cdot)\Vert_{L^\infty(\Omega)}. 
\end{split}
\end{equation}
Since $H_0^1(\Omega)$ embedds continuously in $L^\infty(\Omega)$ due to Rellich-Kondrachov theorem \cite[Theorem 9.16, page 285]{brezis2010functional}, there exists a positive constant $C_\Omega$ such that
\begin{equation}
\Vert B^\star z\Vert_S\leq C_\Omega \Vert a(\cdot)\Vert_{L^\infty(\Omega)} \Vert z_t(t,\cdot)\Vert_{H^1_0(\Omega)}.
\end{equation}
Noticing that
\begin{equation}
\Vert (z,z_t)\Vert_{D(A)}:=\Vert z\Vert_{H^2(\Omega)\cup H^1_0(\Omega)} + \Vert z_t\Vert_{H^1_0(\Omega)},
\end{equation}
then
\begin{equation}
\Vert B^\star z\Vert_S \leq C_\Omega \Vert a(\cdot)\Vert_{L^\infty(\Omega)} \Vert (z,z_t)\Vert_{D(A)}. 
\end{equation}
This implies that \eqref{continuous-embedding-exp} holds for \eqref{wave-equation-linear} with $c_S:=C_\Omega\Vert a(\cdot)\Vert_{L^\infty(\omega)}$. Hence, one can apply Theorem \ref{thm-abstract-lyap}. In particular, the origin of \eqref{wave-equation-nonlinear} is semi-globally exponentially stable in $D(A)$ for any nonlinear damping function $\sigma$ satisfying all the items of Definition \ref{def-sat}.

\begin{remark}
In \cite{MVC}, a similar result is provided for damped  wave equations in dimension $N\leq 3$. The strategy the authors follow in the paper does not rely on a Lyapunov functional, but rather on a analysis of the natural energy of the wave equation in order to obtain an integral inequality. Note that their result are better than ours for the wave equation because they need an  $L^\infty$ bound (unifrom in time along a trajectory) only for $z$ while we (essentially) need a similar bound for $z_t$. 
\end{remark}

\section{Conclusion}
\label{sec_conclusion}

In this paper, we have characterized the asymptotic behavior of a family of linear infinite-dimensional systems subject to a nonlinear damping. Assuming that the origin of the system is globally exponentially stable or globally polynomially stable when the damping is linear, we have built Lyapunov functionals for the nonlinear system. These Lyapunov functionals are the sum of two terms: the first one is based on the Lyapunov operator coming the stabilizability property of the linear system and the second term is added in order to compensate the nonlinearities.

From this work, there exist many research lines which can be pursued further. Below, we have listed some of them.
\begin{itemize}
\item Unfortunetaly, our strategy in the case where $S\neq U$ (i.e.,  $S=L^\infty(\Omega)$) does not work for the wave equation with a dimension higher than $2$. It might be interesting to investigate a weaker property than \eqref{continuous-embedding-exp} in order to characterize precisely the asymptotic behavior of the wave equation subject to a nonlinear damping;
\item In some papers (see e.g., \cite{MVC}), the nonlinear damping function is not assumed to be maximal monotone (i.e., it does not satisfy Item 2. of Definition \ref{def-sat}). We believe that our general strategy might also work without assuming such a property, focusing on some particular partial differential equations.
\item It might be also interesting to investigate ISS properties of such linear infinite-dimensional systems subject to a nonlinear damping. The case where $S=U$ has been tackled in \cite{mcp2018ECC}, but the case $S\neq U$ seems harder to obtain.
\item Our strategy might be also adapted for other nonlinearities, such as the dry damping, which has been studied for instance in \cite{ervedoza2018dry}. In contrast with the nonlinear damping introduced in this paper, the dry damping is not smooth (it is described with a sign function), which makes the well-posedness study of the closed-loop system not trivial to tackle as well as an asymptotic behavior characterization.
\end{itemize}

\label{app-finite}

\section{Lyapunov functions for linear finite-dimensional systems subject to a nonlinear damping}
\label{sec_finite}

\subsection{Deriving Lyapunov functions for the finite-dimensional case}

Let us consider the following linear finite-dimensional system
\begin{equation}
\label{finite}
\frac{d}{dt} z = Az + Bu,
\end{equation}
where $z\in\mathbb{R}^n$, $u\in\mathbb{R}^m$, and $A$ and $B$ of appropriate dimension. Let us denote by $|\cdot |$ the Euclidian norm of $\mathbb{R}^n$ and $|B|$ the induced norm of the matrix $B$. We use $^\top$ to denote the transpose of a matrix. We suppose that the  following properties hold:
\begin{itemize}
\item[(i)] the pair $(A,B)$ is controllable;
\item[(ii)] $A$ is dissipative, i.e., for every $z\in \mathbb{R}^n$,
\begin{equation*}
z^\top A z + z^\top A^\top z\leq 0. 
\end{equation*}
\end{itemize}
Then, for every $k>0$, the feedback-law $u=-kB^\top z$ stabilizes the origin of \eqref{finite}, i.e. the matrix $A-kBB^\top$ is Hurwitz. This means in particular that there exists a unique symmetric positive-definite matrix $P\in\mathbb{R}^{n\times n}$ such that
\begin{equation}
\label{finite-lyapunov-ineq}
(A-kBB^\top)^\top P + P(A-kBB^\top)= - I_{\mathbb{R}^n}.
\end{equation}

We aim at modifying the feedback control $u=-kBB^\star z$ with a nonlinear damping function given by Definition \ref{def-sat} and at building a strict Lyapunov function for the corresponding nonlinear system. Note that in this case $S=U=R^m$. This Lyapunov function is based on the following function
\begin{equation}
\label{lyap-S=U}
K: X\in\mathbb{R}_+\mapsto \int_0^{X} \sqrt{v} h(|B| \sqrt{v})dv\in\mathbb{R}_+.
\end{equation}
In particular, this function is positive, strictly increasing, vanishes at $0$ and tends to infinity as $X$ tends to infinity. 
%This implies that this function has an inverse, that we denote by $K^{-1}$. 

\begin{theorem}
\label{thm-finite}
Consider a nonlinear damping function only satisfying Items $1.$ and $3.$ of Definition \ref{def-sat}, where $S=U=\mathbb{R}^m$. Let $P$ be the solution of \eqref{finite-lyapunov-ineq} with $k=C_1$ provided in \eqref{eq:inf-sat}. Then, the positive definite function $\tilde V:\mathbb{R}^n\to\mathbb{R}_+$ given by
\begin{equation}\label{eq:tildeV:finite}
\tilde V(z):= z^\top P z +C_2 |B|\ | P|K(\vert z\vert^2),
\end{equation}
where $C_2$ is provided in \eqref{eq:inf-sat}, is a strict Lyapunov function for the following nonlinear system
\begin{equation}
\label{nonlinear-finite}
\frac{d}{dt} z = Az - B\sigma(B^\top z):=A_\sigma(z),
\end{equation}
and, along its trajectories, one has
\begin{equation}
\label{abs-strict-lyap}
\frac{d}{dt} \tilde{V}(z) \leq - |z|^2.
\end{equation}
\end{theorem} 

\begin{proof}
Set $\tilde{A}=A-C_1BB^\star$, where the positive constant $C_1$ is given in \eqref{eq:inf-sat} and $P$ is defined in \eqref{finite-lyapunov-ineq}. Consider the following candidate Lyapunov function $\tilde V(z):= z^\top P z +MK(\vert z\vert)$
where $M$ is a sufficiently large positive constant to be chosen later and $h$ is the function defined in Item 2 of Definition \ref{def-sat}. This function, inspired by \cite{liu1996finite}, is positive definite and coercive. Indeed, since $h(0)>0$, 
%there exists a positive constant $C_h$ such that $h(0)\geq C_h$. Therefore, using the fact that $h$ is increasing, the following holds
%$$
%h(\Vert B\Vert_{\mathbb{R}^{m\times n}}\sqrt{v})\geq C_h,\quad \forall v\in [0,|z|^2 ].
%$$
%The latter implies
\begin{equation}
\int_{0}^{\vert z\vert^2} \sqrt{v} h(|B|\sqrt{v}) dv\geq h(0) \int_{0}^{|z|^2} \sqrt{v} dv = \frac{2h(0)}{3} |z|^3.
\end{equation}
Therefore, $\tilde{V}\geq \frac{2h(0)}{3} | z|^3+\lambda_{\min}(P)| z|^2$, where $\lambda_{\min}(P)$ is the smallest eigenvalue of $P$. It implies in particular that the Lyapunov function $\tilde{V}$ is coercive. Moreover, noticing that $z^\top P z\leq | P| |z|^2$ and that $h$ is increasing, one has therefore
\begin{equation}
\label{Lyapunov-semi-bounds}
\lambda_{\min}(P)| z|^2+\frac{2Mh(0)}{3} \vert z\vert^3 \leq \tilde{V}\leq |P| | z|^2 + M|z|^3 h(|B||z|).
\end{equation}
Applying Cauchy-Schwarz's inequality, one has
\begin{align*} 
\frac{d}{dt}z^\top P z = & z^\top P\tilde{A}z + z^\top P\tilde{A}^\top z +2 z^\top P B(C_1B^\top z-\sigma(B^\top z)\nonumber\\
%\leq & -C|z|^2+2 z^\top B^\top P C_1B^\top z-\sigma(B^\top z)\label{eq:SU}\\
\leq &-|z|^2+2C_1|B^\top Pz|\ |B^\top z-\sigma(B^\top z)|\nonumber\\
\leq &- |z|^2 + 2C_2|B|\ |P|\ | z| h(| B^\top z|)
z^\top B^\top\sigma(B^\top z),
\end{align*}
%for some positive constant $C$.
where we have used in the last inequality Item 3. of Definition \ref{def-sat}.
%, it yields
%\begin{align*}
%\frac{d}{dt} z^\top Pz\leq &- |z|^2 + 2C_2|B^\top ||P|| z| h(| B^\top z|)
%z^\top B^\top\sigma(B^\top z)\\
%\leq & -|z|^2 + 2C_2|B^\top|| P|| z| h(| B|| z|)
%z^\top B \sigma(B^\top z).
%\end{align*}
Secondly, using the dissipativity of the matrix $A_\sigma$, one has
\begin{align*}
M\frac{d}{dt} \int_0^{|z|^2}  \sqrt{v} h(|B|\sqrt{v})dv 
=& M|z| h(|B|\ |z|)(z^\top A^\top z
 + z^\top Az - 2z^\top B\sigma(B^\top z))\\
\leq & -2M|z| h(|B|\ |z|)z^\top B\sigma(B^\top z).
\end{align*}
Hence, if one chooses $M=C_2 |B|\ | P|$,
one obtains, after adding the two above inequalities, the desired inequality \eqref{abs-strict-lyap}. This achieves the proof of Theorem \ref{thm-finite}.
\end{proof}

\begin{remark}\label{remFin1}
Damping functions are usually of the type $z\mapsto
(\sigma_i(z_i))_{1\leq i\leq m}$, where $\sigma_i$ is a real valued
damping function. If $\sigma(z)/z\to 0$ as $\vert z\vert$ tends to
infinity, then $\sigma$ is said to be a weak damping function, for
instance $C_1s/(1+\vert s\vert)^k$, with $C_1,k>0$. In this case, up
to a positive constant, the function $h$ can be taken equal to one if
$k\geq 1$ and to $(1+\xi)^{k-1}$ if $k\geq 1$. If moreover
$\sigma(\cdot)$ admits non zero limits at infinity, then $\sigma$ is
sometimes called a
saturation function, for instance $\arctan(s)$ or $\tanh(s)$ and, in
this case, the function $h$ is equal to a positive constant. The
definition of damping function is (essentially) known for saturation
functions (cf. \cite{liu1996finite}, \cite{sussmann1991saturation}),
especially the key inequality \eqref{eq:inf-sat}, in which case
the admissible function $h$ is simply constant.%\startCP probleme de label, je ne sais pas de quelle inegalite Yacine voulait parler\stopCP
\end{remark}
\begin{remark}\label{remFin2} The behavior of a damping function at
infinity is rather general but the behavior at zero is linear. In the
case of a real valued damping function, the results of this paper can
be easily generalized
to the case where
\begin{equation}\label{eq:b-at-0}
0<\liminf_{z\to 0}\frac{\sigma(z)}z\leq \limsup_{z\to
0}\frac{\sigma(z)}z<\infty.
\end{equation}
The only modification occurs in \eqref{eq:inf-sat}, %\startCP probleme de label, je ne sais pas de quelle equation Yacine voulait parler\stopCP 
where the constant $C_1$ must be replaced by a positive function $C_1(\cdot)$ bounded below and
above by two positive constants.
\end{remark}

\subsection{Asymptotic behavior characterization}
%\begin{remark}
\label{rem-asymptotic-behavior}
We claim that, once a trajectory enters the unit ball, then it converges exponentially to the origin. Indeed, let $t^\star$ the time such that:
\begin{equation}
\vert W_\sigma(t^\star)z_0\vert = 1, 
\end{equation}
where $t^\star=0$ if  $|z_0|\leq1$.
Since $(W_\sigma(t))_{t\geq 0}$ is a strongly continuous semigroup of contractions, one has
\begin{equation}
\vert W_\sigma(t)z_0\vert \leq \vert W_{\sigma}(t^\star)z_0\vert \leq 1,\quad t\geq t^\star.
\end{equation}

Note that, for all $t\geq t^\star$, one has $\vert W_\sigma(t)z_0\vert^3\leq \vert W_\sigma(t) z_0\vert^2$. Therefore, since $h$ is increasing, \eqref{Lyapunov-semi-bounds} reduces in the unit ball
to 
\begin{equation}
 \lambda_{min}(P)|z|^2\tilde{V}(z) \leq (|P|+ M h(|B|))|z|^2
\end{equation}
and \eqref{abs-strict-lyap} 
\begin{align*}
\frac{d}{dt} \tilde{V}(z) \leq & - C_V\tilde{V}(z), \ \forall t\geq t^*,
\end{align*}
where $C_V:=\frac{1}{|P|+ M h(|B|)}$. Then, one gets easily the claim.
%Applying Gr\"onwall inequality, using \eqref{Lyapunov-semi-bounds}, the fact that $h$ is increasing and that $(W_\sigma(t))_{t\geq 0}$ is a strongly continuous semigroup of contractions, it follows
%\begin{align*}
%\tilde{V}(W_\sigma(t-t^\star) z_0)\leq &e^{-C_V(t-t^\star)}V(W_\sigma(t^\star)z_0)\\
%\leq &e^{-C_V(t-t^\star)} \left(|P||W_\sigma(t^\star)z_0|^2 + M|W_\sigma(t^\star z_0)|^2h(|B|)\right)\\
%\leq &(|P|+Mh(|B|))e^{-C_V(t-t^\star)}|z_0|^2.
%\end{align*}
%Using again \eqref{Lyapunov-semi-bounds}, one has
%\begin{equation}
%|W_\sigma(t-t^\star)z_0|^2 \leq \frac{|P|+Mh(|B|)}{\lambda_{min}(P)}e^{-C_V(t-t^\star)} |z_0|^2.
%\end{equation}
%Finally, since $(W_\sigma(t))_{t\geq 0}$ is a strongly continuous semigroup of contractions, one has, for all $t\geq t^\star$
%\begin{align*}
%|W_\sigma(t)z_0|^2\leq &|W_\sigma(t-t^\star)z_0|^2 \\
%\leq &\frac{(|P|+Mh(|B|))e^{C_Vt^\star}}{\lambda_{min}(P)}e^{-C_V t} |z_0|^2.
%\end{align*}

Hence, it remains to characterize the behavior of trajectories of \eqref{nonlinear-finite} before they enter the unit ball. The function $X\mapsto K(X) + \lambda_{\min}(P) X$ is strictly increasing and hence defines a bijection from $\mathbb{R}_+$ to $\mathbb{R}_+$. It has a strictly increasing inverse function, that we call $g$. Then, along any trajectory of \eqref{nonlinear-finite},
\begin{equation}
\frac{d}{dt} \tilde{V}(z) \leq - g(\tilde{V}(z)).
\end{equation}
From here, one can characterize the asymptotic behavior of \eqref{nonlinear-finite}: there exist two positive constants $C_3,C_4$ such that, for $|z_0|$ large enough,
\begin{equation}
|z(t)|\leq C_3\sqrt{(g\circ G)}(C_4 |z_0|-t),\quad \forall t\in \left[0,C_3|z_0|-1\right],
\end{equation}
where $G$ is the function defined by 
\begin{equation*}
G(|z|):=\int_1^{|z|}\frac{dv}{g(v)}.
\end{equation*}
For instance, if the nonlinear damping function $\sigma$ is given by any saturation function satisfying \eqref{satLinfty}, then $h$ is the identity, and we have
 \begin{equation*}
 \label{ineq1-opti}
|z(t)|\leq C_3(C_4 |z_0|-t),\quad \forall t\in \left[0,C_4|z_0|-1\right],
\end{equation*}
which is a linear decay of the trajectories with large initial conditions. It is actually optimal since one can prove a converse inequality as follows. Since $A$ is dissipative, one has that
$$
\frac{d}{dt}\vert z\vert^2=2z^\top \frac{d}{dt} z=-2z^\top B\sigma(B^\top z)\geq -2C_\sigma\vert B\vert\vert z\vert,
$$
where $C_\sigma$ is a constant bounding $\sigma$. Therefore, for all $t\geq 0$
\begin{equation}
\frac{d}{dt}\vert z\vert  \geq -2C_\sigma\vert B\vert. 
\end{equation}
This implies that
\begin{equation}
\label{ineq2-opti}
\vert z\vert\geq -2C_\sigma\vert B\vert t - \vert z_0\vert.
\end{equation}
Hence, for a suitable positive constant $C_5$ and $C_6$ depending on $C_3$, $C_4$, $C_\sigma$ and $\vert B\vert$ and for a sufficiently large initial condition, one therefore has
\begin{equation}
\vert z\vert =C_5(C_6\vert z_0\vert - t), \forall t\in [0,C_6|z_0|-1]. 
\end{equation} 

\bibliographystyle{plain}
\bibliography{bibsm}

\end{document}